\documentclass[11pt, eqno]{article}
\usepackage{bbm}
\usepackage{mathrsfs}
\usepackage{amsfonts}
\usepackage{amssymb}
\usepackage{graphicx}
\usepackage[all]{xy}
\usepackage{amsthm}
\usepackage{amsmath}
\usepackage{amsmath,amssymb,latexsym,color}
\usepackage[mathscr]{eucal}
\usepackage{CJK}
\usepackage{cases}
\usepackage{graphics}
\usepackage{psfrag}
\usepackage{subfigure}
\usepackage{color}

\usepackage{amssymb,latexsym}
\usepackage{amsmath,latexsym}
\usepackage{amscd}
\newcommand{\N}{{\mathbb N}}

\newcommand{\R}{{\mathbb R}}

\newcommand{\Z}{{\mathbb Z}}

\newtheorem{theorem}{Theorem}[section]

\newtheorem{definition}[theorem]{Definition}

\newtheorem{remark}[theorem]{Remark}

\newtheorem{lemma}[theorem]{Lemma}

\newtheorem{proposition}[theorem]{Proposition}
\newtheorem{claim}[theorem]{Claim}

\begin{document}

\title{Higher $P$-symmetric  Ekeland-Hofer capacities}
\date{February 1, 2021}
\author{Kun Shi and Guangcun Lu
\thanks{Corresponding author
%\endgraf \hspace{2mm}
%Partially supported by the Fundamental Research Funds
%for the Central Universities  at BNU.
%eijing Normal University. %by the NNSF  11271044 of China.
\endgraf\hspace{2mm} 2020 {\it Mathematics Subject Classification.}
52B60, 53D35, 52A40, 52A20.}}
 \maketitle \vspace{-0.3in}

%This last installment concludes this series of papers.
%Õâ¸ö×îºó²¿·Ö¶Ô±¾ÏµÁÐÎÄÕÂ×öÁËÒ»¸ö×ܽᡣ

%This is the first paper in the series
%of studies on Landau-Ginzburg models in the contexts of the mirror symmetric
%and other topics

\abstract{
This paper is devoted to the construction of analogues of higher Ekeland-Hofer symplectic capacities for $P$-symmetric
subsets in the standard symplectic space $(\mathbb{R}^{2n},\omega_0)$, which is motivated by  Long and Dong's study $P$-symmetric 
closed characteristics on $P$-symmetric convex bodies.
%We construct a analogues of higher Ekeland-Hofer symplectic capacities for $P$-symmetric closed characteristics instead of
%closed characteristics motivated by Long and Dong's work.
 We study the relationship between these capacities and other capacities, and give some computation examples.
 Moreover, we also define higher real symmetric Ekeland-Hofer capacities as a complement of Jin and the second named author's
 recent study  of the  real symmetric analogue about the first Ekeland-Hofer capacity.
} \vspace{-0.1in}
\medskip\vspace{12mm}

\maketitle \vspace{-0.5in}

%%%%%%%%%%%%%%%%%%%%%%%%%%%%%%%%%%%%%%%%%%%%%%%%%%%%%%%%%%%%%%%%%%%%%%%%%%%%%%%%%%%%%%%%%%%%%%%%%%%%%%%%%%%%%%%%%%%%%%%%%%
%%\noindent{\it Keywords}: Ekeland-Hofer symplectic capacity; Hofer-Zehnder symplectic capacity; Brunn-Minkowski type inequality; Minkowski billiard \vspace{2mm}
%%%%%%%%%%%%%%%%%%%%%%%%%%%%%%%%%%%%%%%%%%%%%%%%%%%%%%%%%%%%%%%%%%%%%%%%%%%%%%%%%%%%%%%%%%%%%%%%%%%%%%%%%%%%%%%%%%%%%%%%%%
%Critical groups;

%\tableofcontents

\section{Introduction}
\setcounter{equation}{0}

Motivated by studies of closed characteristics, Ekeland and Hofer \cite{EH89, EH90} introduced the concept of
the symplectic capacities and constructed a sequence of symplectic capacities $c_{EH}^j$,
 nowdays called the Ekeland-Hofer (symplectic) capacities.
They are symplectic invariants for subsets in the standard symplectic space $(\mathbb{R}^{2n},\omega_0)$, and play important
actions in symplectic topology and Hamiltonian dynamics. Recently, in \cite{JL19,JL20, JL19*}
Rongrong Jin and the second named author gave several generalizations of the first Ekeland-Hofer capacity $c_{EH}^1$
and studied their properties and applications.

Our purpose of this paper is to define analogues of higher Ekeland-Hofer symplectic capacities
for subsets with some kind of symmetric in  $(\mathbb{R}^{2n},\omega_0)$.
Without special statements, we always use $J_0$ to denote standard complex structure on $\mathbb{R}^{2n}$,
and $\langle \cdot, \cdot\rangle_{\mathbb{R}^{2n}}$ standard inner product on $\mathbb{R}^{2n}$.
Define $P={\rm diag}(-I_{n-\kappa},I_\kappa,-I_{n-\kappa},I_\kappa)$ for some integer $\kappa\in[0,n)$. A subset $A\subset \mathbb{R}^{2n}$
is said to $P$-\textsf{symmetric} if $PA=A$, that is, $x\in A$ implies $Px\in A$.
Let $\mathcal{B}(\mathbb{R}^{2n})=\{B\subset\mathbb{R}^{2n}\,|\, PB=B\;\&\;B\cap {\rm Fix}(P)\neq\emptyset\}$.

Recall that the  \textsf{distinguished line bundle} of a compact smooth connected hypersurface $\mathcal{S}$ in
the standard symplectic space $(\mathbb{R}^{2n},\omega_0)$ is
defined by
$$
\mathcal{L}_{\mathcal{S}}=\{(x,\xi)\in T\mathcal{S}\,|\,\omega_{0x}(\xi,\eta)=0, \forall \eta\in T_x\mathcal{S}\}.
$$
A $C^1$ embedding $z:\mathbb{R}/T\mathbb{Z}\rightarrow \mathcal{S}$ is called a \textsf{closed characteristic} on $\mathcal{S}$ if
$\dot{z}(t)\in(\mathcal{L}_{\mathcal{S}})_{z(t)}$ for all $t\in \mathbb{R}/T\mathbb{Z}$.
When the hypersurface $\mathcal{S}$ is $P$-symmetric, %i.e., $x\in\mathcal{S}$ implies $Px\in\mathcal{S}$,
such a closed characteristic $z$ on $\mathcal{S}$ is said to be
$P$-\textsf{symmetric} if it also satisfies $z(t+\frac{T}{2})=Pz(t)$  for all $t\in \mathbb{R}/T\mathbb{Z}$.
Similarly, if $\mathcal{S}$ is only the boundary of convex body $D$ in $(\mathbb{R}^{2n},\omega_0)$, we call
 a nonconstant absolutely continuous curve  $x:[0,T]\rightarrow \partial D$ (for some $T>0$) to be a \textsf{generalized closed characteristic} on $\mathcal{S}$ if $x(0)=x(T)$
 and  $\dot{x}(t)=J_0N_{\mathcal{S}}(x(t)), \hbox{a.e.}$, where
$N_{\mathcal{S}}(x)=\{y\in\mathbb{R}^{2n}\,|\,\langle u-x,y\rangle_{\mathbb{R}^{2n}}\leqslant0, \forall u\in D\}$; when $D$ is also $P$-symmetric,
%that is, $PD=D$,
the generalized closed characteristic $x$ on $\mathcal{S}$ is said to be \textsf{$P$-symmetric} if $x(t+\frac{T}{2})=Px(t)$ for all $x\in[0,T]$.
As usual the action of a (generalized) closed characteristic $x$ is defined by
$$
\mathcal{A}(x)=\frac{1}{2}\int_0^T\langle -J_0\dot{x},x\rangle dt.
$$
 Let us write
\begin{eqnarray*}
&&\Sigma_{\mathcal{S}}=\{k\mathcal{A}(x)>0\,|\, x \hbox{ is a  closed characteristic on $\mathcal{S}$},\;k\in\N\},\\
&&\Sigma_{\mathcal{S}}^P=\{k\mathcal{A}(x)>0\,|\, x \hbox{ is a $P$-symmetric closed characteristic on $\mathcal{S}$},\;k\in\N\}.
\end{eqnarray*}
Hereafter $\N$ always denotes the set of all positive integers.
They are called the \textsf{action spectrum} and the \textsf{$P$-symmetric action spectrum} of $\mathcal{S}$ or $D$, respectively.

%
%Here, we consider the following closed characteristics:
%\begin{definition}
%{\rm \begin{description}
%\item[(i)] For a smooth hypersurface $\mathcal{S}$ in a symplectic manifold $(\mathbb{R}^{2n},\omega_0)$, a $C^1$ embedding $z:\mathbb{R}/T\mathbb{Z}\rightarrow \mathcal{S}$ is called a {\bf $P$-symmetric closed characteristic} on $\mathcal{S}$ if $z(t+\frac{T}{2})=Pz(t)$ and $\dot{z}(t)\in(\mathcal{L}_{\mathcal{S}})_{z(t)}$ for all $t\in \mathbb{R}/T\mathbb{Z}$.
%\item[(ii)] If $\mathcal{S}$ is the boundary of convex body $D$ in $(\mathbb{R}^{2n},\omega_0)$, we say a nonconstant absolutely continuous curve  $x:[0,T]\rightarrow \partial D$(for some $T>0$)to be a generalized closed characteristic on $\mathcal{S}$ if $\dot{x}(t)=J_0N_{\mathcal{S}}(x(t)), a.e.$  and $x(0)=x(T)$, where
%$N_{\mathcal{S}}(x)=\{y\in\mathbb{R}^{2n}|\langle u-x,y\rangle\leqslant0, \forall u\in D\}$. If $D$ is $P$-symmetric,that is $PD=D$,
%we call a generalized closed characteristic on $\mathcal{S}$ is $P$-symmetric if $x(t+\frac{T}{2})=Px(t), \forall x\in[0,T]$.
%\end{description}}
%\end{definition}

Dong and Long \cite{DL041} first studied the existence of $P$-symmetric closed characteristics on
boundaries of $P$-symmetric convex bodies in $\mathbb{R}^{2n}$.
%This leads to our study in this paper, that is,
By considering such closed characteristics
 we define analogues of higher Ekeland-Hofer symplectic capacities,
 $$
 c_P^j:\mathcal{B}(\mathbb{R}^{2n})\to [0, +\infty],\quad j=1,2,\cdots,
 $$
  which are called  $P$-\textsf{symmetric higher Ekeland-Hofer symplectic capacities}.
  As in \cite{EH89, EH90} the following proposition can be proved easily.

  %, denoted by  And give some results about this symplectic capacity.
%For convenience,
%Let
%$$
%\mathcal{B}(\mathbb{R}^{2n})=\{B\subset\mathbb{R}^{2n}\,|\, PB=B\;\&\;B\cap {\rm Fix}(P)\neq\emptyset\}.
%$$
%be the set of $P$-symmetric subset $B\subset\mathbb{R}^{2n}$ with $B\cap {\rm Fix}(P)\neq\emptyset$.

%The existence of this kind of characteristic was first considered by Dong and Long in \cite{DL041}.
%In \cite{HL14}, Hui Liu proved that there exist $2$ geometrically distinct $P$-symmetric closed characteristics on $\partial{D}$, when $n\geqslant2$ and $D$ is $P$-symmetric convex domain in $\mathbb{R}^{2n}$.

%Notice the $P$-symmetric closed characteristics on $\partial D$ are in one-to-one correspondence  with solutions $x:[0,\frac{T}{2}]\rightarrow \partial D$ of
%\begin{equation*}
%\left\{
%   \begin{array}{l}
%     \dot{x}=J_0N_{\partial D}(x(t)),\\
%     x(\frac{T}{2})=Px(0);\\
%   \end{array}
%   \right.
%\end{equation*}
%The solution of this kind of equation was studied in more general case by Rongrong Jin and Guangcun Lu \cite{JL19} in more general case. They define the $\Psi$-Hofer-Zehnder symplectic capacity and $\Psi$-Ekeland-Hofer symplectic capacity for general $\Psi\in{\rm Sp}(2n)$.

\begin{proposition}\label{prop:property}
$\forall j\in\mathbb{N}$, $c_{P}^j$ %:\mathcal{B}(\mathbb{R}^{2n})\rightarrow [0,\infty]$
has the following property:
\begin{description}
\item[(i)(Monotonicity)]If $B_1,B_2\in\mathcal{B}(\mathbb{R}^{2n})$ and $B_1\subset B_2$, then $c_{P}^j(B_1)\leqslant c_{P}^j(B_2)$;
\item[(ii)(Conformality)] $c_P^j(\lambda B)=\lambda^2c_{P}^j(B),\forall\lambda\in\mathbb{R},\forall B\in\mathcal{B}(\mathbb{R}^{2n})$;
\item[(iii)] $c_{P}^j(h(B))=c_{P}^j(B)$ for each $P$-equivariant
 $h\in{\rm Symp}(\mathbb{R}^{2n})$.
 % is $P$-equivariant, then
%$$c_{P}^j(h(B))=c_{P}^j(B);$$
\item[(iv)]$c_{P}^j:\mathcal{B}(\mathbb{R}^{2n})\rightarrow\mathbb{R}$ is continuous with respected to Hausdroff distance on $\mathcal{B}(\mathbb{R}^{2n})$.
\end{description}
\end{proposition}

%Denote
%$$\Sigma_{S}^P=\{\mathcal{A}(x)>0|x\text { is a } P\text{-symmetric characteristic on }  S\}.$$

For the first $P$-Ekeland-Hofer symplectic capacity $c_P^1$, we have:

\begin{theorem}\label{th:min}
For a $P$-symmetric convex bounded domain $D\subset\mathbb{R}^{2n}$ with $C^{1,1}$ boundary $\mathcal{S}=\partial D$,
if it contains a fixed point of $P$, then
$$
c_P^1(D)=c_P^1(\mathcal{S})=\min\Sigma_\mathcal{S}^P.
$$
%%
%Let $D\subset\mathbb{R}^{2n}$ be a $P$-symmetric convex bounded domain with $C^{1,1}$ boundary $S=\partial D$ and contain a fixed point of $P$. Then
%$$c_P^1(D)=\min\Sigma_S^P.$$
%Moreover, if both $\partial D$ and $D$ contain fixed points of $P$, then
%$$c_P^1(D)=c_P^1(\partial D).$$
\end{theorem}

As arguments below \cite[Remark~1.12]{JL20}  the condition ``$C^{1,1}$" for $\mathcal{S}$  can be removed if
we  replace $\min\Sigma_\mathcal{S}^P$ by
$\min\{\mathcal{A}(x)>0\,|\, x \hbox{ is a $P$-symmetric generalized characteristic on $\mathcal{S}$}\}$.

%In the case of non-smooth, we can get the corresponding results for $P$-symmetric generalized characteristic.

A compact smooth connected hypersurface $\mathcal{S}$ in $\mathbb{R}^{2n}$ is said to be of {\it restricted contact type}
 if there exists a Liouville vector field $\eta$ on $\mathbb{R}^{2n}$ such that $\mathcal{L}_{\eta}\omega_0=\omega_0$
 and $\eta$ points transversely outward at $\mathcal{S}$. Denote $B_{\mathcal{S}}$ by the bounded component of $\mathbb{R}^{2n}\setminus \mathcal{S}$, and by
%A closed characteristic on $\mathcal{S}$ is an embedded circle $P\subset\mathcal{S}$ satisfying $TP=\mathcal{L}_{\mathcal{S}}|_P$,
%where $\mathcal{L}_{\mathcal{S}}\rightarrow \mathcal{S}$ is a distinguished line bundle
%Consider the space Denote by
$$\mathcal{L}(\mathbb{R}^{2n})=\{B_{\mathcal{S}}\subset\mathcal{B}(\mathbb{R}^{2n})\,|\,\mathcal{S} \text{ is a hypersurface of restricted contact type}\}.$$
Proposition~\ref{prop:property}(i) implies that each $c_P^j(B_\mathcal{S})$ is always finite.
%For higher $P$-Ekeland-Hofer symplectic capacity, we have the following result:

\begin{theorem}\label{th:hig}
%If $B_{\mathcal{S}}\in\mathcal{L}(\mathbb{R}^{2n})$ has finite the j-th $P$-Ekeland-Hofer  capacity, i.e., $c_P^j(B_\mathcal{S})<\infty$,
%then
$c_P^j(\mathcal{S})=c_P^j(B_\mathcal{S})\in \Sigma^P_{\mathcal{S}}$.
%In particular, if $\Delta$ also contains a fixed point of $P$, then
%$$c_P^j(\Delta)=c_P^j(B_{\Delta}).$$
\end{theorem}

The following theorem gives  relationships between the Ekeland-Hofer symplectic capacities, the generalized Ekeland-Hofer-Zehnder symplectic capacities and the $P$-symmetric Ekeland-Hofer symplectic capacities
with  $P={\rm diag}(-I_{n-\kappa},I_\kappa,-I_{n-\kappa},I_\kappa)$ for some integer $\kappa\in[0,n]$.

\begin{theorem}\label{th:relation}
If $D\in\mathcal{B}(\mathbb{R}^{2n})$ is a bounded convex domain, then
\begin{eqnarray}
&&c_P^1(D)=c_{EH}^1(D)\quad\hbox{for $\kappa=0$},\label{e:relation1}\\
%&&c_P^1(D)=2c_{EH}^1(D)\quad\hbox{for $\kappa=n$},\\
&&c_P^1(D)=2c_{\rm EHZ}^P(D)\quad\hbox{for $1\le\kappa< n$},\label{e:relation2}
\end{eqnarray}
%$$
%c_P^1(B)=c_{EH}^1(B);$$
%If $k=n$, $$c_P^1(B)=2c_{EH}^1(B).$$
%If $1\leqslant k\leqslant n$, $c_P^1(B)=2c_{\rm EHZ}^P(B)$,
where $c_{\rm EHZ}^P$ is the generalized Ekeland-Hofer-Zehnder symplectic capacities defined in \cite{JL19}.
\end{theorem}

If $\kappa>0$, it is possible that (\ref{e:relation1}) fails. See an counterexample
in Remark~\ref{rm:relation}(ii) below.

\begin{proof}[Proof of Theorem~\ref{th:relation}]
%\begin{proof}
By an approximation as in \cite[Section~4.3]{JL20} we can assume that $D$ is strictly convex and has a $C^\infty$ boundary $\mathcal{S}$.
Then
 $$
 c_{EH}^1(D)=\min\{\mathcal{A}(x)>0\,|\, x\text{ is a  closed characteristic on }\mathcal{S}\}.
 $$
 This and Theorem \ref{th:min} lead to $c_{EH}^1(D)\le c_{P}^1(D)$.

If $\kappa=0$, then $D=-D$. By \cite[Corollary 2.2]{AK17} any closed characteristic of minimal action on the boundary $\mathcal{S}$ is itself centrally
symmetric. Hence
$$
 c_{EH}^1(D)=\min\{\mathcal{A}(x)>0\,|\, x\text{ is a centrally
symmetric closed characteristic on }\mathcal{S}\}.
 $$
Moreover, in the present case a closed characteristics on $\mathcal{S}$
is $P$-symmetric  if and only if it is central symmetric. It follows from these and Theorem \ref{th:min} that  $c_P^1(D)=c_{EH}^1(D)$.
%$$
%c_{P}^1(D)\ge \min\{\mathcal{A}(x)>0\,|\, x\text{ is a central symmetric closed characteristic on }\mathcal{S}\}.
%$$
%Moreover,

Let $1\leqslant \kappa< n$ and let $x:[0,T]\rightarrow \mathcal{S}$ be
a $P$-symmetric closed characteristics  on $\mathcal{S}$. Since
 %satisfies $\dot{x}(t)\in J_0N_{\mathcal{S}}(x(t))$  and
 $x(t+\frac{T}{2})=Px(t)$,  $y:=x|_{[0,\frac{T}{2}]}$
  is a $P$-characteristic on $\mathcal{S}$ in the sense of \cite[Definition~1.1]{JL19}
    and $\mathcal{A}(x)=2\mathcal{A}(y)$.
 Conversely, if $y:[0,T]\rightarrow\mathcal{S} $ is a $P$-characteristic on $\mathcal{S}$, then
\begin{equation*}
x(t)=\left\{
    \begin{array}{l}
     y(t), \text{ if } t\in[0,T],\\
     y(t-T), \text{ if } t\in[T, 2T]\\
   \end{array}
   \right.
\end{equation*}
is a $P$-symmetric closed characteristics on $\mathcal{S}$, and $\mathcal{A}(x)=2\mathcal{A}(y)$.
Thus by these two claims, and \cite[Theorem~1.9]{JL19} and Theorem~\ref{th:min},
we deduce that $c_P^1(D)\le 2c_{\rm EHZ}^P(D)$ and $c_P^1(D)\ge 2c_{\rm EHZ}^P(D)$, respectively.
             %we can immediately get the result what we want.
\end{proof}

%We also give some specify examples:
In the following we consider computations of the $P$-Ekeland-Hofer  capacities of ellipsoids and polydiscs.
Let $r=(r_1,\cdots,r_n)$ with $r_i>0$ for each $i=1,\cdots,n$. We call sets
$$
E(r):=\{z\in\mathbb{R}^{2n}|\sum_{i=1}^n\frac{x_i^2+y_i^2}{r_i^2}<1\}
\quad\hbox{and}\quad
D(r):=B^2(r_1)\times\cdots\times B^2(r_n)
$$
the ellipsoid and the polydisc of radius $r$, respectively.
Define a set
$$
\sigma_P(r)=\{m\pi r_j^2\,|\,m\in2\mathbb{N}+1,\; j=1,\cdots,n-\kappa\}\bigcup\{m\pi r_j^2\,|\, m\in2\mathbb{N}, \;j=n-\kappa+1,\cdots, n\}
$$
and a map
$$
\phi:\mathbb{N}\times\{1,\cdots,n\}\rightarrow \sigma_P(r),\;(k,j)\mapsto kr_j.
$$
The multiplicity of $d\in\sigma_P(r)$ is defined by
$$
m(d):=\sharp\phi^{-1}(d).
$$
 From $\sigma_P(r)$ we construct a nondecreasing sequence of numbers $\{d_i\}_i$, where
$$
 d_i=d_i(\sigma_P(r)),\; i=1,2,\cdots,
$$
such that  each $d\in\sigma_P(r)$ is  repeated  $m(d)$ times.
%For the ellipsoid
%$$
%E(r):=\{z\in\mathbb{R}^{2n}|\sum_{i=1}^n\frac{x_i^2+y_i^2}{r_i^2}<1\}.
%$$
%We have

\begin{theorem}\label{th:E}
$c_P^j(E(r))=d_j(\sigma_P(r))$ for each $j\in\N$.
\end{theorem}

For $r=(1,\cdots,1)$, we immediately obtain
\begin{equation*}
c_P^j(B^{2n}(1))=\left\{
    \begin{array}{ll}
     \pi\; & \text{ if } j=1,\cdots,n-\kappa,\\
     2\pi\; & \text{ if } j=n-\kappa+1,\cdots, n,\\
     3\pi\; &\text{ if } j=n+1,\cdots, 2n-\kappa,\\
     4\pi\;& \text{ if } j=2n-\kappa+1,\cdots, 2n,\\
     \cdots\\
   \end{array}
   \right.
\end{equation*}

\begin{remark}\label{rm:relation}
{\rm For $r=(r_1,\cdots,r_n)$, let $\bar{r}=(r_{n-\kappa+1},r_2,\cdots,r_{n-\kappa},r_1,r_{n-\kappa+2},\cdots,r_n)$. \\
{\bf (i)} If $r_1\neq r_{n-\kappa+1}$, then the corresponding sequence $\{\bar{d}_i\}_i$ to $\sigma_P(\bar{r})$ is different from
the sequence $\{d_i\}_i$ of $\sigma_P({r})$ though it is possible that sets $\sigma_P(\bar{r})$ and $\sigma_P({r})$ coincide.
Hence Theorem~\ref{th:E} implies that there is no $P$-equivariant symplectic diffeomorphisms from $E(r)$ to $E(\bar{r})$.\\
{\bf (ii)} If $\kappa>0$ and $\min\{r_1,\cdots,r_{n-\kappa}\}>2\min\{r_{n-\kappa+1},\cdots, r_n\}$, we can directly compute
$$
c_P^1(E(r))=2\min\{r_{n-\kappa+1},\cdots, r_n\}\quad\hbox{and}\quad c_{EH}^1(E(r))=\min\{r_{n-\kappa+1},\cdots, r_n\}
$$
by Theorem~\ref{th:E} and \cite[Proposition~4]{EH90}, respectively, and so
 $c_P^1(E(r))\neq c_{EH}^1(E(r))$.}
 \end{remark}

For $r=(r_1,\cdots,r_n)$, define
%$$D(r):=B^2(r_1)\times\cdots\times B^2(r_n).$$
%and
$$
\hat{r}=\min\{r_i\,|\,i=1,\cdots,n-\kappa\}\quad\hbox{and}\quad r'=\min\{r_i\,|\,i=n-\kappa+1,\cdots,n\}.
$$
Let $\sigma_P'(r)=\{(2m-1)\pi\hat{r}^2, 2m\pi r'^2\;|\;m\in\mathbb{N}\}$. By a similar construction to $\{d_i(\sigma_P(r))\}_i$, we can
get a sequence of numbers $\{d_i(\sigma_P'(r))\}_i$.

\begin{theorem}\label{th:D}
$c_P^j(D(r))= d_j(\sigma_P'(r))$.
\end{theorem}

 Since $E(r)\subset D(r)$, $c_P^j(D(r))\geqslant c_P^j(E(r))>0$ by the monotonicity.
 For $j=1$, noting that  $\min\sigma_P(r)=\pi\min\{\hat{r}^2, 2r'^2\}$ and $\min\sigma'_P(r)=\pi\min\{\hat{r}^2, 2r'^2\}$
 by the definitions of $\sigma_P(r)$ and $\sigma'_P(r)$, we have $c_P^1(D(r))= c_P^1(E(r))$.

Consider the Lagrangian bidisk in
$(\mathbb{R}^{4}(x_1,x_2,y_1,y_2), \omega_0)$ with $\omega_0=\sum_{i=1}^2dx_i\wedge dy_i$ (\cite{Ra17}),
$$
D^2\times_L D^2=\{(x_1,x_2, y_1,y_2)\in\mathbb{R}^{4}\,|\,x_1^2+x_2^2<1,\, y_1^2+y_2^2<1\}
$$
By \cite[Proposition 2.2]{BMP19} the action sprectrum of %closed characteristic on
$\partial(D^2\times_L D_2)$ is
$$
\Sigma_{\partial(D^2\times_L D^2)}=\{2n\cos(\theta_{k,n})\,|\,k,n\in\N,\,\theta_{k,n}\in J_n\}\cup\{2n\pi\,|\, n\in\N\}
$$
where $J_n=\{(2k-1)\pi/2n\,|\, 1\le k\le (n-1)/2\}$ if $n$ is odd, and
$J_n=\{k\pi/n\,|\, 0\le k\le n/2-1\}$ if $n$ is even.
In general, it is difficult to determinate $\Sigma_{\partial(D^2\times_L D^2)}^P$.
It was shown that
$c_{EH}^1(D^2\times_L D^2)=4$ (\cite{AO14,BMP19}) and
$c_{EH}^2(D^2\times_L D^2)=3\sqrt{3}$ and $c_{EH}^3(D^2\times_L D^2)=8$
(\cite{Ra17, BMP19}).
By (\ref{e:relation1}) $c_P^1(D^2\times_L D^2)=c_{EH}^1(D^2\times_L D^2)=4$ if
$P={\rm diag}(-1,-1,-1,-1)$. For another $P$ we have:

\begin{theorem}\label{th:lag}
If $P={\rm diag}(-1,1,-1,1)$, then $c^2_P(D^2\times_L D^2)\in\{2\pi,8\}$ and thus
there exists a $P$-symmetric generalized closed characteristic $x$  with action $\mathcal{A}(x)\in\{2\pi,8\}$ on the $\partial(D^2\times_LD^2)$.
But for $c_P^1(D^2\times_LD^2)$ we can only get the estimate:
$$
c_P^1(D^2\times_LD^2)\in[4,4\pi\frac{\sqrt{41}-4}{5})\cap\Sigma_{\partial (D^2\times_L D^2)}.
$$
\end{theorem}

  A real symplectic manifold is a triple $(M,\omega,\tau)$ consisting of a symplectic manifold $(M,\omega)$ and an anti-symplectic involution $\tau$ on $(M,\omega)$, i.e. $\tau^*\omega=-\omega$ and $\tau^2=id_M$.
  The standard symplectic space $(\mathbb{R}^{2n},\omega_0)$ is real with respect to the canonical involution $\tau_0:\mathbb{R}^{2n}\to
  \mathbb{R}^{2n}$ given by $\tau_0(x,y)=(x,-y)$.
  For $\tau_0$-invariant subsets in $(\mathbb{R}^{2n},\omega_0)$ Jin and the second named author defined
 a symmetrical version of the first Ekeland-Hofer capacity of such subsets in \cite{JL20}.
   In Section~\ref{sec:real}, we will also give  symmetrical versions of the higher  Ekeland-Hofer capacities of such subsets.

  % a higher Ekeland-Hofer capacity, as a generalization of Ekeland-Hofer capacity in real symplectic manifold. This is a complement of Jin and Lu's work about real symplectic manifold in \cite{JL20}.
%    We only give the definition in form, not give too much discussion.

%\noindent{\bf Organization of the paper}.
The paper is organized in the following way. In Section~\ref{sec:3}
we first present our variational frame and then give variational definitions of the higher $P$-symmetric  Ekeland-Hofer capacities.
Section~\ref{sec:4} proves Theorem~\ref{th:E} and Theorem~\ref{th:D}.
In Sections~\ref{sec:5},~\ref{sec:6} and \ref{sec:7} we give proofs of
Theorems~\ref{th:min},~\ref{th:hig} and \ref{th:lag}, respectively.

% The arrangements of the paper is as follows. \\
% $\bullet$ Section~\ref{sec:space} gives our variational frame and related preparations.\\
%  $\bullet$ Section~\ref{sec:EH.2} discusses the variational explanation for our extended  Ekeland-Hofer capacity $c^\Psi_{\rm EH}$.\\
% $\bullet$ Section~\ref{sec:convex}  proves Theorem~\ref{th:convex}
%by improving the arguments in \cite{HoZe87, HoZe90}.\\
%$\bullet$ Section~\ref{sec:EH.3} proves Theorems~\ref{th:EHconvex},~\ref{th:EHproduct}.\\
%$\bullet$ Section~\ref{sec:EH.4} proves Theorem~\ref{th:EHcontact}.\\
%$\bullet$ Section~\ref{sec:appl} proves  Theorems~\ref{MaSch}, \ref{Str}.\\
%$\bullet$ Section~\ref{sec:convexDiff} proves Theorem~\ref{th:convexDiff}.\\
%$\bullet$ Section~\ref{sec:Brunn} proves Theorem~\ref{th:Brun} and Corollary~\ref{cor:Brun.2}.\\
%$\bullet$ Finally, Section~\ref{sec:BillT} proves Theorems~\ref{th:billT.2},~\ref{th:billT.3}.\\

\section{Variational definitions of higher $P$-symmetric  Ekeland-Hofer capacities}\label{sec:3}
\setcounter{equation}{0}

Our method follows \cite{EH90} basically. The variational frame
is slight modifications of that of \cite{EH90}. For clearness and completeness
we state necessary definitions and  results.
% and then
%we first present our variational frame and then give variational definitions of the higher $P$-symmetric  Ekeland-Hofer capacities.
%

\subsection{Variational frame}\label{var}

Let $S^1=\R/Z$ and let $X_{P}=\{x\in L^2(S^1;\mathbb{R}^{2n})\,|\, x(t+\frac{1}{2})=Px(t)\,a.e. t\in\R\}$.
Then $x\in L^2(S^1;\mathbb{R}^{2n})$ sits in $X_{P}$ if and only if
coefficients of its Fourier series $x=\sum_{j\in \mathbb{Z}}e^{2\pi jtJ_0}x_j$
satisfy conditions: $Px_j=-x_j$ for all $j\in 2\Z+1$, and $Px_j=x_j$ for $j\in 2\Z$.
%
%\begin{align*}
%X_{P}&=\{x\in L^2(S^1;\mathbb{R}^{2n})|x(t+\frac{1}{2})=Px(t)\}\\
%&=\{x\in L^2(S^1;\mathbb{R}^{2n})|x=\sum_{j\in \mathbb{Z}}e^{2\pi jtJ_0}x_j,
%Px_j=-x_j, \text{ for } j \text{ is odd};Px_j=x_j,\text{ for } j \text{ is even}\}.\\
%\end{align*}
For $s\ge 0$, consider the Hilbert space
$$
E^s=\left\{x\in L^2(S^1;\mathbb{R}^{2n})\,\Bigm|\,x=\sum_{j\in\mathbb{Z}}e^{2\pi jtJ_0}x_j,\,x_j\in\R^{2n},\,\sum_{j\in\mathbb{Z}}|j|^{2s}|x_j|^2<\infty\right\}
$$
with inner product and associated norm given by
\begin{eqnarray}\label{innerproduct}
&&\langle x,y\rangle_{s}=\langle x_0,y_0\rangle_{\mathbb{R}^{2n}}+ 2\pi\sum_{j\in\mathbb{Z}}
|j|^{2s}|x_j|^2,\nonumber\\
&&\|x\|_{s}^2=\langle x,x\rangle_{s}
\end{eqnarray}
(\cite{EH90}). Then $E^s_P:=E^s\cap X_P$ is a closed subspace of $E^s$. Throughout, we write
\begin{eqnarray}\label{innerproduct+}
E:=E^{1/2}\quad\hbox{and}\quad E_P:=E\cap X_P.
\end{eqnarray}
%
%Following \cite{EH90} let $E$ denote the Hilbert space consisting of all $x\in L^2(\R/\Z,\mathbb{R}^{2n})$
%For $s\geqslant 0$, define a linear subspace of $L^2(S^1,\mathbb{R}^{2n})$ by
%$$E_P^s=\{x\in X_P|\sum_{k\neq 0}|j|^{2s}|x_j|^2<\infty\}.
%$$
%The space $E_P^s$ is a Hilbert space equipped with the inner product
%$$\langle x,y\rangle_{s}:=\langle x_0,y_0\rangle+2\pi\sum_{j\in\mathbb{Z}}|j|^{2s}\langle x_j,y_j\rangle$$
%and the norm $\parallel x\parallel_{s}^2=\langle x,x\rangle_{s}$.
%Denote $E_P:=E_P^{1/2}$.
There exists an orthogonal splitting
$E_P=E_P^-\oplus E_P^0\oplus E_P^+$,
where $E_P^0={\rm Fix}(P)\equiv\mathbb{R}^{2k}$ and
$$
E_P^-=\{x\in E|x=\sum_{j<0}x_j e^{2\pi jtJ_0}\},\quad E_P^+=\{x\in E|x=\sum_{j>0}x_j e^{2\pi jtJ_0}\}.
$$
Denote $P^+, P^-, P^0$ by the orthogonal projections onto $E_P^+, E_P^-, E_P^0$, respectively.
Then every $x\in E_P$ has the unique decomposition
$x=x^-+x^0+x^+$, where $x^+=P^+x, x^0=P^0x$ and $x^-=P^-x$.
%$P={\rm diag}(-I_{n-\kappa},I_\kappa,-I_{n-\kappa},I_\kappa)$
 Define the functional  $\mathcal{A}:E_P\rightarrow\mathbb{R}$ by
$$
\mathcal{A}(x)=\frac{1}{2}\parallel x^+\parallel_{\frac{1}{2}}^2-\frac{1}{2}\parallel x^-\parallel_{\frac{1}{2}}^2
$$
if $x=x^++x^0+x^-$. It is easy to prove that $\nabla\mathcal{A}(x)=x^+-x^-\in E_P$ and
$$
\mathcal{A}(x)=\frac{1}{2}\int_0^1\langle -J_0\dot{x},x\rangle_{\mathbb{R}^{2n}} dt,\quad\forall x\in C^1(S^1,\mathbb{R}^{2n})\cap E_P.
$$
By Propositions~3,4 on pages 84-85 of \cite{HoZe94} we immediately obtain:

\begin{proposition}\label{pro:s}
{\bf (i)} If $t>s\geqslant 0$, the inclusion map $I_{t,s}:E_P^t\rightarrow E_P^s$ is compact.\\
{\bf (ii)} Assume $s>\frac{1}{2}$. Then each $x\in E_P^s$ is continuous and satisfies $x(t+\frac{1}{2})=Px(t)$
for all $t\in\R$.
Moreover, there exists a constant $c>0$ such that
$$\sup_{t\in S^1}|x(t)|\leqslant c\|x\|_s, \forall x\in E_P^s.$$
\end{proposition}
%\begin{proof}
%Let $P_N:E_P^t\rightarrow E_P^s$ be
%$$(P_Nx)(t):=\sum_{|j|\leqslant N}e^{2\pi jtJ_0}x_j \text{ for }x(t)=\sum_{j\in\mathbb{Z}}e^{2\pi jt J_0}x_j.$$
%It's a continuous linear operator and the image of $P_N$ is a finite dimensional space. There $P_N$ is a compact operator. In addition,
%\begin{eqnarray*}
%\parallel(P_N-I_{t,s})x\parallel_s^2&=&\parallel\sum_{|j|>N} e^{2\pi jt J_0}x_j\parallel_s^2\\
%&=&2\pi\sum_{|j|>N}|j|^{2s}|x_j|^2\\
%&=&2\pi\sum_{|j|>N}|j|^{2s-2t}|j|^{2t}|x_j|^2\\
%&\leqslant&N^{2(s-t)}2\pi\sum_{|j|>N}|j|^{2t}|x_j|^2\\
%&\leqslant &N^{2(s-t)}\parallel x\parallel_t^2.\\
%\end{eqnarray*}
%Since $t>s$, $\parallel(P_N-I_{t,s})\parallel^{op}\leqslant N^{s-t}\rightarrow 0$
%as $N\rightarrow\infty$. Thus, $I_{t,s}$ is compact.
%\end{proof}

%\begin{proposition}\label{pro:s}
%Assume $s>\frac{1}{2}$. Then each $x\in E_P^s$ is continuous and satisfies $x(t+\frac{1}{2})=Px(t)$
%for all $t\in\R$.
%Moreover, there exists a constant $c>0$ such that
%$$\sup_{t\in S^1}|x(t)|\leqslant c\parallel x\parallel_s, \forall x\in E_P^s.$$
%\end{proposition}
%\begin{proof}
%For $x=x_0+\sum_{j\neq0}e^{2\pi jtJ_0}x_j$, since
%\begin{eqnarray*}
%|x(t)|\leqslant|x_0|+\sum_{j\neq0}|e^{2\pi jtJ_0}x_j|&=&\sum_{j\neq0}|x_j|+|x_0|\\
%&=&\sum_{j\neq0}|j|^{-s}|j^{s}|x_j|+|x_0|\\
%&\leqslant&(\sum_{j\neq0}|j|^{-2s})^{\frac{1}{2}}(\sum_{j\neq0}|j|^{2s}|x_j|^2)^{\frac{1}{2}}+|x_0|\\
%\end{eqnarray*}
%$(\sum_{j\neq0}|j|^{-2s})^{\frac{1}{2}}$ converges as $2s>1$ and $|x_0|\leqslant \parallel x\parallel_s$, so we can get the proposition.
%\end{proof}

Let $j:E_P\rightarrow L^2(S^1;\mathbb{R}^{2n})$ be the inclusion map, and $j^*:L^2(S^1;\mathbb{R}^{2n})\rightarrow E_P$ be the adjoint operator of $j$ defined by
$$
\langle j(x),y\rangle_{L^2}=\langle x,j^*(y)\rangle_{\frac{1}{2}},\; \forall x\in E_P, y\in L^2(S^1;\mathbb{R}^{2n}).
$$
As in the proof of Proposition~5 on page 86 of \cite{HoZe94} we have
\begin{proposition}
For $y\in L^2$, $j^*(y)\in E_P^1$ and satisfies
$\|j^*(y)\|_1\leqslant \|y\|_{L^2}$.
Consequently,  $j^*$ is a compact operator.
\end{proposition}
%\begin{proof}
%\begin{eqnarray*}
%\sum_{m\in\mathbb{Z}}\langle x_m,y_m\rangle&=&(j(x),y)_{L^2}\\
%&=&\langle x,j^*(y)\rangle_{\frac{1}{2}}\\
%&=&\langle x_0,j^*(y)_0\rangle+2\pi \sum_{m\in\mathbb{Z}}|m|\langle x_m,j^*(y)_m\rangle\\
%\end{eqnarray*}
%for all $x\in E_P, y\in L^2$. Hence,
%$$j^*(y)_0=y_0, j^*(y)_m=\frac{y_m}{2\pi|m|},m\neq0 .$$
%Thus,
%$$j^*(y)(t)=y_0+\sum_{m\neq0}e^{2\pi mtJ_0}\frac{1}{2\pi|m|}y_m, y\in L^2,$$
%From the definition of the norm in $E_P^1$, we get $\parallel j^*(y)\parallel_1\leqslant \parallel y\parallel_{L^2}$.
%
%Since $j^*:L^2\rightarrow E_P^1\rightarrow E_P$ and the embedding $I_{1,\frac{1}{2}}:E_P^1\rightarrow E_P$ is compact. So $j^*$ is a compact operator.
%\end{proof}

Let $\hat{\mathcal{H}}(\mathbb{R}^{2n})$ be  the set of nonnegative smooth function $H\in C^{\infty}(\mathbb{R}^{2n})$ satisfying the following condition:
\begin{description}
\item[(H1)] $H(z)=H(Pz), \forall z\in\mathbb{R}^{2n}$;
\item[(H2)] There is an open set $U\subset\mathbb{R}^{2n}$ such that $H|_U\equiv 0$;
\item[(H3)] When $|z|$ is large enough, $H(z)=a|z|^2$, where $a\in (\pi, \infty)\setminus\mathbb{N}\pi$.
\end{description}

For $H\in\hat{\mathcal{H}}(\mathbb{R}^{2n})$, we have a number $M>0$ such that
$|H(z)|\leqslant M|z|^2$ and $|d^2H(z)|\leqslant M$ for all $z\in\mathbb{R}^{2n}$.
Define  functionals $b_H, \mathcal{A}_H:E_P\rightarrow \mathbb{R}$ by
$$
b_H(x)=\int_0^1 H(x(t))dt, x\in L^2\quad\hbox{and}\quad \mathcal{A}_H(x):=\mathcal{A}(x)-b_H(x).
$$
These two functionals are also well-defined on $E$.
Note that the gradient $\nabla b_H(x)\in E_P$ is equal to the orthogonal projection onto $E_P$
of the gradient of $b_H$ at $x$ as a functional on $E$.
By Lemma~4 on page 87 of \cite{HoZe94} we have:
\begin{proposition}
The map $b_H:E_P\rightarrow \mathbb{R}$ is continuously differentiable. The gradient $\nabla b_H:E_P\rightarrow E_P$ is continuous and compact.
Moreover,  for all $x,y\in E_P$ there hold
$$
\|\nabla b_H(x)-\nabla b_H(y)\|_{\frac{1}{2}}\leqslant M\|x-y\|_{\frac{1}{2}}\quad\hbox{and}\quad |b_H(x)|\leqslant M\|x\|_{L^2}^2.
$$
\end{proposition}

\begin{proposition}\label{pro:reg}
Each critical point $x$ of $\mathcal{A}_H$ in $E_P$ belongs to
 $C^{\infty}(S^1,\mathbb{R}^{2n})$ and satisfies
$$
\dot{x}(t)=J_0\nabla H(x(t))\quad\hbox{and}\quad x(t+\frac{1}{2})=Px(t), \;\forall t\in S^1.
$$
\end{proposition}

\begin{proof}
Since $\nabla b_H(x)=j^*(\nabla H(x))$,     %$x\in E_P$ is a critical point of $\mathcal{A}_H$.
we have $x^+-x^-=j^*(\nabla H(x))$ and so
\begin{equation}\label{e:reg}
\langle x^+-x^-,v\rangle_{\frac{1}{2}}=\langle j^*\nabla H(x),v\rangle_{\frac{1}{2}}=(\nabla H(x),j(v))_{L^2}
\end{equation}
For $v=\sum_{j\in\mathbb{Z}}e^{2\pi jtJ_0}v_j\in E_P$. Substitute
 the Fourier series of $x$ and $\nabla H$,
$$
x=\sum_{j\in\mathbb{Z}}e^{2\pi jtJ_0}x_k\quad\hbox{and}\quad \nabla H(x)=\sum_{j\in\mathbb{Z}} e^{2\pi jtJ_0}a_k,
$$
into (\ref{e:reg}), we obtain $a_j=2\pi jx_j$ for all $j\in\mathbb{Z}$.
It follows that
$$
a_0=0=\int_0^1\nabla H(x(t))dt\quad\hbox{and}\quad
\sum |j|^2|x_j|^2\leqslant\sum|a_j|^2=\|\nabla H(x)\|_{L^2}^2<\infty.
$$
This implies that $x\in E_P^1$, and therefore  $x$ is continuous by Proposition \ref{pro:s}(ii).
Others are repeating of the proof of Lemma~5 on page 88 of \cite{HoZe94}.
%Thus, $\nabla H(x(t))$ is continuous. Define $\xi(t)=\int_0^t J\nabla H(x(s))ds$. Then $\xi\in C^1(S^1,\mathbb{R}^{2n})$. In fact, it's easily to compute that the Fourier coefficient $\xi_j$ for $j\neq0$ satisfies $\xi_k=x_k$.
%Therefore, from the uniqueness of the Fourier coefficients and $\xi(0)=0$,
%$$\xi(t)=x(t)-x(0).$$
%Thus $x\in C^1(S^1,\mathbb{R}^{2n})$ and satisfying $\dot{x}(t)=J_0\nabla H(x(t))$.
%Therefore, $\dot{x}=J_0\nabla H(x(t))$ is $C^1$, and $x$ is $C^2$. By induction, we can get $x$ is smooth.
\end{proof}

 In the proof of Lemma~6 on page 89 of \cite{HoZe94} replacing $E$ by $E_P$ leads to:

\begin{proposition}\label{pro:PS}
If $H\in\hat{\mathcal{H}}(\mathbb{R}^{2n})$, then every sequence $\{x_j\}_j\subset E_P$ with $\nabla\mathcal{A}_H(x_j)\rightarrow 0$ has a convergent subsequence. In particular, $\mathcal{A}_H$ satisfies (PS) condition.
\end{proposition}

Similarly, the proof of Lemma~7 on page 90 of \cite{HoZe94} yields
\begin{proposition}\label{pro:rep}
The flow $\phi^t(x)$ of the gradient equation $\dot{x}=-\nabla\mathcal{A}_H(x)$ on $E_P$ has representation
$$\phi^t(x)=e^tx^-+x^0+e^{-t}x^++K(t,x)$$
for all $t\in\mathbb{R}$ and $x=x^-+x^0+x^+\in E_P$, where $K:\mathbb{R}\times E_P\rightarrow E_P$ is continuous and maps bounded sets to precompact sets. In addition, $K(t,\cdot):E_P\rightarrow E_P$ is $P$-equivariant, for all $t\in\mathbb{R}$.
\end{proposition}
%\begin{proof}
%Define
%$$K(t,x):=\int_0^t(e^{t-s}P^-+P^0+e^{-t+s}P^+)\nabla b(\phi^s(x))ds.$$
%For $x\in E_P$,
%let $y(t)=e^tx^-+x^0+e^{-t}x^++K(t,x)$. Then
%$$\dot{y}(t)=((P^+-P^-)y)(t)+\nabla b(\phi^t(x)).$$
%On the other hand,
%$$\frac{\phi^t(x)}{dt}=(P^+-P^-)\phi^t(x)+\nabla b(\phi^t(x)).$$
%Therefore
%$$\frac{d}{dt}(y(t)-\phi^t(x))=(P^+-P^-)(y(t)-\phi^t(x)), \text{ and } y(0)=\phi^0(x)=x.$$
%we get $y(t)=\phi^t(x)$.
%Note that
%$$K(t,x)=j^*\int_0^t(e^{t-s}P^-+P^0+e^{-t+s}P^+)\nabla H(j(\phi^s(x)))ds.$$
%Thus, $K$ is the composition of a continuous,bounded map and compact map $j^*$, so $K$ is a compact and continuous map.
%\end{proof}

%\subsection{The definition of higher $P$-symmetric  Ekeland-Hofer capacities}
\subsection{Defining higher $P$-symmetric  Ekeland-Hofer capacities}\label{sec:2.2}

Consider the natural $S^1$-action on $E_P$:
$\theta\star x(t)=x(\theta+t),\; \forall\theta,t\in S^1$.
Let $\mathcal{E}$ be the collection of $S^1$-invariant subsets of $E_P$.
As in \cite{FR78} we can assign a \textsf{Fadell-Rabinowitz index }
$$
i_{S^1,\alpha}(X)=\sup\{m\in\mathbb{N}\,|\,f^*(\alpha^{m-1})\neq 0\}
$$
 to each nonempty $X\in \mathcal{E}$, where $\alpha\in H^2(\mathbb{C}P^{\infty};\mathbb{Q})$ is the Euler class of the classifying vector bundle
$ES^1=S^{\infty}\rightarrow BS^1=\mathbb{C}P^{\infty}$, and $f^*: H^*(\mathbb{C}P^{\infty};\mathbb{Q})\rightarrow H_{S^1}^*(X;\mathbb{Q})$
is the homomorphism induced by the classifying map
$f:(X\times S^{\infty})/S^1\rightarrow\mathbb{C}P^{\infty}$ given by
$f([x, s])=[s]$.
We also define $i_{S^1,\alpha}(\emptyset)=0$ for convenience. Then $i_{S^1,\alpha}$ satisfies the properties
in \cite[Theorem~5.1]{FR78}. Since ${\rm Fix}(S^1)=E_P^0={\rm Fix}(P)\equiv\mathbb{R}^{2k}$,  \cite[Corollary~7.6]{FR78}
showed that $i_{S^1,\alpha}$ has the $2$-dimension property, that is,
$i_{S^1,\alpha}(V^{2m}\cap S_P)=m$ for each $S^1$-invariant $2m$-dimension subspace $V^{2m}$ of $E_P$ such that $V^{2m}\cap{\rm Fix}(S^1)=\{0\}$,
where $S_P=\{x\in E_P\,|\,\|x\|_{1/2}=1\}$.
% the dimension property becomes

 %
%
%From \cite{FR78}, we can also get
%\begin{proposition}
%$i_{S^1,\alpha}$ is a normalized $S^1$-index on $\Sigma$ and satisfies $2$-dimension property.
%\end{proposition}

Consider the group $\Gamma$ of homeomorphisms $h:E_P\to E_P$ of form
$$
h(x)=e^{\gamma^+(x)}P^+(x)+P^0(x)+e^{\gamma^-(x)}P^-(x)+K(x),
$$
where
i) $K:E_P\rightarrow E_P$ is a $S^1$-equivariant continuous map, and maps bounded set to precompact set;
ii) $\gamma^+,\gamma^-:E_P\rightarrow \mathbb{R}^+$ is $S^1$-invariant continuous function, and maps bounded set to bounded set;
iii) there exists a constant $c>0$ such that $\gamma^+(x)=\gamma^-(x)=0$ and $K(x)=0$
for each $x\in E_P$ satisfying $\mathcal{A}(x)\leqslant 0$ or $\parallel x\parallel_{1/2}\geqslant c$.

%\begin{description}
%\item[$\bullet$] $K:E_P\rightarrow E_P$ is a $S^1$-equivariant continuous map, and maps bounded set to precompact set;
%\item[$\bullet$] $\gamma^+,\gamma^-:E_P\rightarrow \mathbb{R}^+$ is $S^1$-invariant continuous function, and maps bounded set to bounded set;
%\item[$\bullet$] there exists a constant $c>0$ such that
%if $x\in E_P$ satisfies $\mathcal{A}(x)\leqslant 0$ or $\parallel x\parallel_{1/2}\geqslant c$,then  $\gamma^+(x)=\gamma^-(x)=0$ and $K(x)=0$
%\end{description}
Following \cite[Definition1.2]{VB82}  we define the pseudoindex $i_{S^1,\alpha}^*$ of $i_{S^1,\alpha}$ relative to $\Gamma$ by
\begin{equation}\label{e:pseudoindex}
i_{S^1,\alpha}^*(\xi):=\inf\{i(h(\xi)\cap S_P^+)\,|\, h\in\Gamma\}\quad\forall\xi\in\mathcal{E},
\end{equation}
where  $S_P^+:=E_P^+\cap S_P$. Repeating the proof of \cite[Proposition~1]{EH90} we can obtain:
%We have the following result about $i_{S^1,\alpha}^*(\xi)$.

\begin{proposition}\label{prop:dim}
Suppose that a $S^1$-invariant subspace $X$ of $E_P^+$ has even dimension ${\rm dim}(X)=2p$. Then
$i_{S^1,\alpha}^*(E_P^-\oplus E_P^0\oplus X)=p$.
\end{proposition}

Let $B\subset\mathbb{R}^{2n}$ be bounded, $P$-symmetric and disjoint with ${\rm Fix}(P)$. For each $H\in\hat{\mathcal{H}}(\mathbb{R}^{2n})$
we define
\begin{equation}\label{e:H-cap}
c_P^i(H):=\inf\{\sup\mathcal{A}_H(\xi)\,|\,\xi\in\Sigma,\, i_{S^1,\alpha}^*(\xi)\geqslant i\},\quad i=1,2,\cdots.
\end{equation}
Clearly, for $H, K\in\hat{\mathcal{H}}(\mathbb{R}^{2n})$ we have
\begin{description}
\item[(i)]{\rm (\textsf{Monotonicity})}  $c_P^i(H)\ge c_P^i(K)$ if $H\le K$.
\item[(ii)] {\rm (\textsf{Continuity})} $|c_P^i(H)-c_P^i(K)|\le \sup\{|H(z)-K(z)|\,\big|\, z\in\mathbb{R}^{2n}\}$.
\item[(iii)] {\rm (\textsf{Homogeneity})}  $c_P^i(\lambda^2H(\cdot/\lambda))=\lambda^2 c_P^i(H)$ for $\lambda\ne 0$.
\end{description}
%(cf. \cite[Prop.3.2.1]{Sik90}).

\begin{proposition}\label{prop:H-cap}
For a given $H\in\hat{\mathcal{H}}(\mathbb{R}^{2n})$, if the associated $a\in (\pi, \infty)\setminus\mathbb{N}\pi$ is as in (H3),
and $j\in\N$ satisfies $a\in(j\pi,(j+1)\pi)$, then for some $\beta>0$ there hold
$$
0<\beta\leqslant c_P^1(H)\leqslant c_P^2(H)\leqslant\cdots\leqslant c_P^m(H)<+\infty,
$$
where $m$ is  half of the dimension of $X_j=\{x\in E_P^+\,|\,x_k=0 \text{ for } k>j\}$, which is equal to
 $nj$ if $j$ is even and $n(j-1)+2(n-\kappa)$ if $j$ is odd (by a direct computation).
\end{proposition}

%$Px_j=-x_j$ for all $j\in 2\Z+1$, and $Px_j=x_j$ for $j\in 2\Z$.$P={\rm diag}(-I_{n-\kappa},I_\kappa,-I_{n-\kappa},I_\kappa)$

\begin{proof}
By the definition of $H$ there exists a $P$-symmetric open set $U\subset\mathbb{R}^{2n}$
 such that  $H=0$ in $U$. Taking $x_0\in U\cap {\rm Fix}(P)$ and arguing as in \cite{EH90} or \cite{JL20},
 we can find $\epsilon>0$ so small that
 \begin{equation}\label{e:beta}
 \mathcal{A}_H|_{x_0+\epsilon S_P^+}\geqslant\beta>0
 \end{equation}
 for some $\beta>0$.  Pick $h\in\Gamma$ such that $h(S_P^+)=x_0+\epsilon S_P^+$.
Let $\xi\in\mathcal{E}$ satisfy  $i_{S^1,\alpha}^*(\xi)\geqslant 1$.
Since $\Gamma$ is a group, $h^{-1}\in\Gamma$. By the definition of $i_{S^1,\alpha}^*(\xi)$ we have
$i_{S^1,\alpha}(h^{-1}(\xi)\cap S_P^+)\ge i_{S^1,\alpha}^*(\xi)\ge 1$
and hence  $h^{-1}(\xi)\cap S_P^+\neq\emptyset$ or equivalently
 $\emptyset\neq\xi\cap h(S_P^+)=\xi\cap(x_0+\epsilon S_P^+).$
It follows from this and (\ref{e:beta}) that $c_{P}^1(H)\geqslant \beta>0$.
By the monotonicity $c_P^k(H)\geqslant c_{P}^1(H)\geqslant \beta>0$ for any integer $k\ge 1$.

Next we claim
\begin{equation}\label{e:beta1}
\sup\mathcal{A}_H(E_P^-\oplus E_P^0\oplus X_j)<+\infty.
\end{equation}
In fact, for any $x\in E_P^-\oplus E_P^0\oplus X_j$, since $\pi\sum_{i=1}^j |i||x_i|^2\leqslant\pi j\sum_{i=1}^j |x_i|^2\leqslant \pi j\int_0^1|x(t)|^2dt$,
 \begin{equation}\label{e:beta2}
 \mathcal{A}_H(x)\leqslant \pi\sum_{i=1}^j |i||x_i|^2-\int_0^1 H(x(t)) dt\le \int_0^1\pi j|x(t)|^2- H(x(t))dt.
 \end{equation}
 By the definition of $H$  there exists a $R>0$ such that  $H(z)=a|z|^2$
 and so $\pi j|z|^2-H(z)\leqslant 0$ for each $z\notin \overline{B^{2n}(R)}$.
 Moreover, by the compactness of $\overline{B^{2n}(R)}$ we have a constant $C>0$ such that
 $\pi j|z|^2-H(z)\leqslant C$ for all $z\in\overline{B^{2n}(R)}$.
  Thus  $\pi j|z|^2-H(z)\leqslant\max\{C,0\}$ for all $z\in\mathbb{R}^{2n}$.
  From these and (\ref{e:beta2}) we derive that $\mathcal{A}_H(x)\leqslant \int_0^1|x(t)|^2- H(x(t))dt\leqslant\max\{C,0\}$
  and so (\ref{e:beta1}).

Since ${\rm dim}{X_j}=nj$ if $j$ is even, and ${\rm dim}{X_j}=(n(j-1)+2(n-\kappa))$ if $j$ is odd, we have
\begin{eqnarray*}
&&i_{S^1,\alpha}^*(E_P^-\oplus E_P^0\oplus X_j)=\frac{nj}{2}\;\hbox{if $j$ is even},\\
&&i_{S^1,\alpha}^*(E_P^-\oplus E_P^0\oplus X_j)=\frac{1}{2}n(j-1)+ (n-\kappa)\;\hbox{if $j$ is odd}
\end{eqnarray*}
by Proposition~\ref{prop:dim}. The desired inequalities follows immediately.
\end{proof}

As in \cite[Lemma~1]{EH90}, using Proposition~\ref{pro:PS}
 we can get:

\begin{proposition}\label{prop:deform}
For $H\in\hat{\mathcal{H}}(\mathbb{R}^{2n})$ and a $S^1$-invariant open neighborhood $\mathcal{N}$ of
$K_c(\mathcal{A}_H)=\{x\in E_P\,|\,\mathcal{A}_H(x)=c, \mathcal{A}'_H(x)=0\}$, there exist $\varepsilon>0$ and $h\in\Gamma$ such that
$$h(\mathcal{A}^{c+\varepsilon}_H\setminus\mathcal{N})\subset\mathcal{A}^{c-\varepsilon}_H,$$
where $\mathcal{A}_H^d=\{x\in E_P|\mathcal{A}_H\leqslant d\}$.
\end{proposition}

As usual, for each $H\in\hat{\mathcal{H}}(\mathbb{R}^{2n})$, if $c_{P}^j(H)<\infty$, then $c_{P}^j(H)$ is critical value of $\mathcal{A}_H$.
%Moreover, for each $k\in\N$ we have $c_{P}^k(H_2)\leqslant c_{P}^k(H_1)$
%if $H_1, H_2\in\hat{\mathcal{H}}(\mathbb{R}^{2n})$ satisfy $H_1\leqslant H_2$.
Let
$$
H\in\mathcal{H}(B):=\{H\in\hat{\mathcal{H}}(\mathbb{R}^{2n})\,|\, H=0\text{ in a neighborhood of } \overline{B}\}
$$
For a $P$-symmetric bounded subset $B\subset\mathbb{R}^{2n}$  with $B\cap{\rm Fix}(P)\neq \emptyset$,
we define the $j$-th $P$-\textsf{symmetric Ekeland-Hofer capacity} by
\begin{equation}\label{e:P-cap}
c_P^j(B):=\inf_{H\in\mathcal{H}(B)}c_{P}^j(H),\quad j=1,2,\cdots.
\end{equation}
If $B\subset\mathbb{R}^{2n}$ is a $P$-symmetric unbounded set with $B\cap{\rm Fix}(P)\neq \emptyset$, we define the $j$-th $P$-\textsf{symmetric Ekeland-Hofer capacity} of it by
\begin{equation}\label{e:P-cap*}
c_P^j(B):=\sup\{c_{P}^j(B')\,|\,B'\subset B\;\hbox{is bounded, $P$-symmetric,  and $B'\cap {\rm Fix}(P)\neq\emptyset$}\}.
\end{equation}

\section{Proof of Theorem~\ref{th:min}}\label{sec:5}
\setcounter{equation}{0}

\begin{proof}[Proof of Theorem~\ref{th:min}]
Let $x_0\in D$ be a fixed point of $P$. Since the translation
$\mathbb{R}^{2n}\to\mathbb{R}^{2n},\;x\mapsto x-x_0$
is a $P$-equivariant symplectomorphism,
by Proposition~\ref{prop:property}(iii)
 we assume $0\in {\rm int}(D)$.
Let $j_D:\mathbb{R}^{2n}\rightarrow \mathbb{R}$ be the Minkowski functional of $D$, and let $H(z)=j^2_D(z)$.

For $\epsilon>0$, let
$\mathcal{F}_{\epsilon}(D)$ consist of $f\circ H$, where $f:[0,+\infty)\rightarrow [0,+\infty)$ satisfies
$$
f(s)=0\, \text{ if } s\leqslant 1,\; f'(s)\geqslant 0\, \text{ if } s\geqslant1,\;f'(s)=\alpha\in\mathbb{R}\setminus\Sigma_S^P\, \text{ if } f(s)\geqslant \epsilon.
$$
Arguing as in the proof of Lemma~6 on page 89 of \cite{HoZe94}, for $f\circ H\in\mathcal{F}_{\epsilon}(D)$
 we can deduce that $\mathcal{A}_{f\circ H}$ satisfies (PS) condition and thus that
$c_P^1(f\circ H)$ is a positive critical value of $\mathcal{A}_{f\circ H}$.
Let $x\in E_P$ be a critical of $\mathcal{A}_{f\circ H}$ with $\mathcal{A}_{f\circ H}(x)>0$.
Then it is nonconstant and satisfies
\begin{equation*}
\left\{
   \begin{array}{l}
     \dot{x}=f'(H(x))J_0\nabla H(x),\\
     x(t+\frac{1}{2})=Px(t).
   \end{array}
   \right.
\end{equation*}
It follows that $y(t):=\frac{1}{\sqrt{s_0}}x(t/T)$ is a solution of
\begin{equation*}
\left\{
   \begin{array}{l}
     \dot{y}=J_0\nabla H(y),\\
     y(t+\frac{T}{2})=Py(t),
   \end{array}
   \right.
\end{equation*}
where $T=f'(s_0)$ and $s_0=H(x(0))$.
A straightforward  computation yields $H(y(t))\equiv 1$ and
\begin{eqnarray*}
\mathcal{A}_{f\circ H}(x)&&=\frac{1}{2}\int_0^1\langle -J_0\dot{x}(t),x(t)\rangle dt-\int_0^1 f\circ H(x(t))dt\\
&&=\frac{1}{2}\int_0^1\langle f'(H(x))\nabla H(x),x(t)\rangle dt-\int_0^1 f\circ H(x(t))dt\\
&&=f'(s_0) s_0-f(s_0),
\end{eqnarray*}
 which lead to $f'(s_0)=\mathcal{A}(y)\in \Sigma_{S}^P$.  By the definition of $f$ we deduce that
 $f(s_0)<\epsilon$ and so
$$
\mathcal{A}_{f\circ H}(x)=f'(s_0)s_0-f(s_0)>f'(s_0)-\epsilon>\min\Sigma_S^P-\epsilon
$$
because $\mathcal{A}_{f\circ H}(x)>0$ implies $s_0=H(x(0))>1$.

Obverse that for any $\epsilon>0$  and $G\in\mathcal{H}(D)$ there exists $f\circ H\in\mathcal{F}_{\epsilon}(D)$ such that $f\circ H\geqslant G$.
We deduce that
$c_P^1(G)\geqslant c_P^1(f\circ H)\geqslant \min\Sigma_S^P-\epsilon$. Hence
$c_P^1(D)\geqslant \min\Sigma_S^P$.

Next,  we  prove  $c_P^1(D)\leqslant \min\Sigma_S^P$.
Let $\gamma:=\min\Sigma_S^P$. It suffices to prove that for any $\epsilon>0$ there exists $\overline{H}\in\mathcal{H}(D)$ such that
\begin{equation}\label{e:gamma}
c_P^1(\overline{H})<\gamma+\epsilon,
\end{equation}
which is reduced to prove that
\begin{equation}\label{e:Omega}
\Omega_h:=\{x\in h(S_P^+)\,|\,\mathcal{A}_{\overline{H}}(x)<\gamma+\epsilon\}\ne\emptyset,\quad\forall
h\in\Gamma.
\end{equation}
In fact, this implies that $\xi_h:=\{\theta\star x\,|\,\theta\in S^1,\,  x\in \Omega_h\}$
is nonempty. Define $\xi=\cup_{h\in\Gamma}\xi_h$.
For $h\in\Gamma$,  take $x\in \Omega_{h^{-1}}$.  Then $x\in h^{-1}(S_P^+)\cap\xi$. This implies $h(x)\in h(\xi)\cap S_P^+$, and thus $i_{S^1,\alpha}^*(\xi)\geqslant 1$.
Note that $\sup_{x\in\xi} \mathcal{A}_{\overline{H}}(x)<\gamma+\epsilon$. (\ref{e:gamma}) holds. %We get $c_P^1(\overline{H})<\gamma+\epsilon$.

It remains to prove (\ref{e:Omega}).
For $\tau>0$, by the definition of $\mathcal{H}(D)$  there exists $H_{\tau}\in\mathcal{H}(D)$ such that
\begin{equation}\label{e:Omega1}
H_{\tau}\geqslant \tau(H-(1+\frac{\epsilon}{2\gamma})).
\end{equation}
As the proof of \cite[Theorem 1.9]{JL19}, for $h\in\Gamma$, we can choose $x\in h(S_P^+)$ such that
$$
\mathcal{A}(x)\leqslant \gamma\int_0^1 H(x(t))dt.
$$
We claim that for $\tau>0$ large enough $\overline{H}:=H_{\tau}$ satisfies (\ref{e:Omega}).
\begin{description}
\item{\bf Case 1.} If $\int_0^1H(x(t))dt\leqslant (1+\frac{\epsilon}{\gamma})$, then by $H_{\tau}\geqslant 0$, we have
    $$
    \mathcal{A}_{H_{\tau}}(x)\leqslant \mathcal{A}(x)\leqslant \gamma\int_0^1 H(x(t))dt\leqslant \gamma(1+\frac{\epsilon}{\gamma})<\gamma+\epsilon.
    $$
\item{\bf Case 2.} If $\int_0^1H(x(t))dt> (1+\frac{\epsilon}{\gamma})$, then
$$
\int_0^1H_{\tau}(x(t))dt\geqslant\tau(\int_0^1H(x(t))dt-(1+\frac{\epsilon}{2\gamma}))\geqslant \tau\frac{\epsilon}{2\gamma}(1+\frac{\epsilon}{\gamma})^{-1}\int_0^1H(x(t))dt$$
because
$$(1+\frac{\epsilon}{2\gamma})=(1+\frac{\epsilon}{2\gamma})(1+\frac{\epsilon}{\gamma})^{-1}(1+\frac{\epsilon}{\gamma})<
(1+\frac{\epsilon}{2\gamma})(1+\frac{\epsilon}{\gamma})^{-1}\int_0^1H(x(t))dt $$
and
$$
1-(1+\frac{\epsilon}{2\gamma})(1+\frac{\epsilon}{\gamma})^{-1}=(1+\frac{\epsilon}{\gamma})^{-1}[(1+\frac{\epsilon}{\gamma})-(1+\frac{\epsilon}{2\gamma})]
=\frac{\epsilon}{2\gamma}(1+\frac{\epsilon}{\gamma})^{-1}.
$$
Pick $\tau>0$ so large that
$$
\tau\frac{\epsilon}{2\gamma}(1+\frac{\epsilon}{\gamma})^{-1}>\gamma.
$$
Then
$$
\int_0^1 H_{\tau}(x(t))dt\geqslant \gamma\int_0^1H(x(t))dt
$$
and hence
$$
\mathcal{A}_{H_\tau}(x)=\mathcal{A}(x)-\int_0^1H_{\tau}(x(t))dt\leqslant \mathcal{A}(x)-\gamma\int_0^1\int_0^1H(x(t))dt\leqslant 0.
$$
\end{description}
In summary, we have $\mathcal{A}_{H_{\tau}}(x)<\gamma+\epsilon$. (\ref{e:Omega}) is proved.

%If both $\partial D$ and $D$ contain fixed points of $P$,
Recall that we have assumed  $0\in D$. Let
$\mathcal{F}_{\epsilon}(\partial D)$
consist of $f\circ H$, where $f:\mathbb{R}\rightarrow \mathbb{R}$ satisfies
$$
f(s)=0 \text{ if } s \text{  near } 1,\; f'(s)\leqslant 0 \text{ if } s\leqslant1,\;
f'(s)\geqslant 0 \text{ if } s\geqslant1,$$
$$f'(s)=\alpha\in\mathbb{R}\setminus\Sigma_S^P \text{ if } f(s)\geqslant \epsilon,$$
and $\alpha>\pi$.
 Repeating the above proof, we can get
 $$
 c_P^1(f\circ H)>\min\Sigma_S^P-\epsilon,\quad \forall f\circ H\in\mathcal{F}_{\epsilon}(\partial D).
 $$
 It follows that $c_P^1(\partial D)\ge \gamma$. By the monotonicity of $c_P^1$, we can get $c_P^1(\partial D)=\gamma$.
\end{proof}

\section{Proof of Theorem~\ref{th:hig}}\label{sec:6}
\setcounter{equation}{0}

\begin{lemma}[\hbox{\cite[Proposition~7.4]{Sik90}}]\label{lem:sik}
Let  $\mathcal{S}$ be a hypersurface of
restricted contact type in $(\mathbb{R}^{2n},\omega_0)$. Then the interior of
$$
\Sigma({\mathcal{S}}):=\{A(x)=\int_x\lambda>0\,|\,x\;\text{is a}\;\hbox{closed characteristic on}\;{\cal S} \}
$$
in $(0, \infty)$ is empty. Moreover, if $\Sigma({\mathcal{S}})$ is nonempty then it contains
a smallest element.
 %i.e., there exists a closed characteristic $x^{\ast}$ on ${\cal S}$
%such that $A(x^\ast)=\min\{A\in\Sigma_{\mathcal{S}}\}$. $\Sigma({\mathcal{S}})$ is a
%closed subset of $\mathbb{R}$.
\end{lemma}

\begin{proof}[Proof of Theorem~\ref{th:hig}]
The ideas are following the proof of \cite[Proposition~6]{EH89}.
Since $\mathcal{S}$ is a $P$-symmetric smooth hypersurface of restricted contact type,  we can
pick a Liouville vector field $\eta$ on $\mathbb{R}^{2n}$ such that $\mathcal{L}_{\eta}\omega_0=\omega_0$ and $\eta$ points
transversely outward at $\mathcal{S}$. Moreover, $\eta$ can be required to be $P$-equivariant in a small neighbor of $\mathcal{S}$
and to have  linear growth. Let $\lambda=\iota_{\eta}\omega_0$, then $d\lambda=\omega_0$.
Thus the  flow $\Psi_{\epsilon}$ generated by $\eta$ is $P$-equivariant in a small neighbor of $\mathcal{S}$.
 Define $\mathcal{S}_{\epsilon}:=\Psi_{\epsilon}(\mathcal{S})$.
 For $\epsilon_0>0$ small enough, without loss of generality, we assume $\epsilon_0=1$ so that
 $\cup_{\epsilon\in(-1,1)}\mathcal{S}_{\epsilon}$ is contained in the small neighborhood.
 It is easy to compute that $\Psi_{\epsilon}^*\omega_0=e^{\epsilon}\omega_0$.

Define $r_0={\rm diam}(\cup_{\epsilon\in[-\frac{1}{2},\frac{1}{2}]}\mathcal{S}_{\epsilon})$ and for $m\in\mathbb{N}$ fix a number $b>(m+\frac{1}{2})\pi r_0^2$.
Pick a smooth map $g:\mathbb{R}\rightarrow\mathbb{R}$ such that
\begin{equation}\label{e:g}
\left.\begin{array}{ll}
&g(s)=b \text{ if } s\leqslant r_0,\;\; g(s)=(m+\frac{1}{2})\pi s^2 \text{ if } s \text{ large },\\
& g(s)\geqslant(m+\frac{1}{2})\pi s^2 \text{ if } s\geqslant r_0,\;\;0<g'(s)\leqslant(2m+1)\pi s \;\hbox{if $s> r_0$}.
 \end{array}\right\}
\end{equation}
In addition, let $\phi:\mathbb{R}\rightarrow\mathbb{R}$ be a smooth map such that for suitable $0<\beta_1<\beta_2<\frac{1}{2}$,
\begin{equation}\label{e:phi}
\phi(s)=0 \text{ if } s<\beta_1,\; \phi'(s)>0 \text{ if }\beta_1<s<\beta_2,\; \phi(s)=b \text{ if } s\geqslant\beta_2.
\end{equation}
 Define a Hamiltonian $H\in\mathcal{H}(B_{\mathcal{S}})$ by
\begin{equation}\label{e:H}
H(z)=\left\{
    \begin{array}{l}
     0\, \text{ if } z\in B_{\mathcal{S}},\\
     \phi(\epsilon)\, \text{ if } z\in\mathcal{S}_{\epsilon},\\
     b\, \text{ if } z\notin B_{\mathcal{S}_{\beta_2}}\,\&\, |z|\leqslant r_0,\\
     g(|z|)\, \text{ if }|z|>r_0.
   \end{array}
   \right.
\end{equation}
Note that a critical point $x$ of $\mathcal{A}_H$ with positive action must satisfy $x([0,1])\subset\mathcal{S}_{\epsilon}$ for some $\epsilon\in(\beta_1,\beta_2)$,
and that $\Sigma_{\mathcal{S}}^P=e^{\epsilon}\Sigma_{\mathcal{S}_{\epsilon}}^P$.

%Firstly, \textsf{we assume that $\Sigma_{\mathcal{S}}^P$ is a discrete set}.

By the definition of $H$, we can see that the positive critical levels only depend on the choice of $\phi$ but not on $g$. In particular $c_P^j(H)$ does not depend on the choice of $g$.
Let %$\mathcal{H}(B_{\mathcal{S}})_{b,m}$ consist of $H\in\mathcal{H}(B_{\mathcal{S}})$ such that
 $$
 \mathcal{H}(B_{\mathcal{S}})_{b,m}=\{H\in \mathcal{H}(B_{\mathcal{S}})\,|\,\hbox{$H(z)\leqslant b$ for $|z|\leqslant r_0$,\; $H(z)=(m+\frac{1}{2})\pi|z|^2$ for $|z|$  large.}\}
 $$
Define
 $$
 c_P^j(B_{\mathcal{S}})_{b,m}=\inf\{c_P^j(H)\,|\,H\in\mathcal{H}(B_{\mathcal{S}})_{b,m}\}.
 $$

 Let  $\hat{\mathcal{H}}(B_{\mathcal{S}})_{b,m}$ consist of $H\in \mathcal{H}(B_{\mathcal{S}})_{b,m}$ as in (\ref{e:H}),
 where $g$ is as in (\ref{e:g}) and $\phi$ as in (\ref{e:phi}) with $0<\beta_1<\beta_2<\frac{1}{2}$ small enough such that
\begin{equation}\label{e:H1}
\phi'(s)\notin\Sigma_{\mathcal{S}_s}^P\;\; \text{ if } \phi(s)\in[\beta_1, b-\beta_2].
\end{equation}
Noting that for any given $H\in\mathcal{H}(B_{\mathcal{S}})_{b,m}$ we can get an $\tilde{H}\in \hat{\mathcal{H}}(B_{\mathcal{S}})_{b,m}$  such that
$\tilde{H}\geqslant H$ (by modifying $g$ if necessary),  there exists  a sequence $\{H_l\}_l\subset \hat{\mathcal{H}}(B_{\mathcal{S}})_{b,m}$ such that
$$
c_P^j(H_l)\rightarrow c_P^j(B_{\mathcal{S}})_{b,m}.
$$
%
% %%%%%%%%%%%%%%%%%%%%%%%%%%%%%%%%%%%%%%%%%%%%%%%%%%
%Let us choose $0<\beta_1^l<\beta_2^l<\frac{1}{2}$ such that $\phi_l$  defined by (\ref{e:phi}) satisfies
%\begin{equation}\label{e:H1}
%\phi'_l(s)\notin\Sigma_{\mathcal{S}_s}^P\;\; \text{ if } \phi_l(s)\in[\beta_1^l, b-\beta_2^l].
%\end{equation}
%Take $g_l$ as in (\ref{e:g}) with $m=l$. Then we can get $H_l$ as in (\ref{e:H}).
%
%Moreover,  for any given $H\in\mathcal{H}(B_{\mathcal{S}})_{b,m}$ we can get an $H_l$ such that
%$H_l\geqslant H$ (by modifying $g_l$ if necessary). It follows that
%$$
%c_P^j(H_l)\rightarrow c_P^j(B_{\mathcal{S}})_{b,m}.
%$$
Taking  critical points $x_l$ of $\mathcal{A}_{H_l}$ with action $\mathcal{A}_{H_l}(x_l)=c_P^j(H_l)$  we have
$x_l([0,1])\subset\mathcal{S}_{\epsilon_l}$ for some $\epsilon_l\in(\beta_1^l,\beta_2^l)$, and therefore
$$
\phi_l(\epsilon_l)\in [0,\beta_1^l) \text{ or } \phi_l(\epsilon_l)\in (b-\beta_2^l,b]
$$
because of (\ref{e:H1}).
Note that $\mathcal{A}(x_l)=\phi_l'(\epsilon_l)$. We get
$$
|c_P^j(H_l)-\mathcal{A}(x_l)|\leqslant \beta_1^l \quad\text{ or }\quad |c_P^j(H_l)-\mathcal{A}(x_l)+b|\leqslant \beta_2^l.
$$
Since the restricted contact condition implies that $\lambda$ does not vanish outside of the zero section of $\mathcal{L}_{\mathcal{S}}$,
%we get $\lambda(x,\xi)\neq0$, for all $(x,\xi)\in\mathcal{L}_{\mathcal{S}}, \xi\neq0$.
we have
$$
c=:\min\bigl\{\frac{|\lambda(x,\xi)|}{|\xi|}\,\big|\,(x,\xi)\in\mathcal{L}_{\mathcal{S}}, |\xi|=1\bigr\}>0,
$$
and thus
$|\lambda(x,\xi)|\ge c|\xi|$ for all $(x,\xi)\in\mathcal{L}_{\mathcal{S}}$.
Moreover, the condition of the restricted contact type is  $C^1$-open, we can get for $\epsilon_0>0$ small enough,
$$
|\lambda(x,\xi)|\ge c/2|\xi|,\quad \forall (x,\xi)\in\mathcal{L}_{\mathcal{S_{\epsilon}}},\quad \forall \epsilon\in [0, \epsilon_0].
$$
It follows that ${\rm length}(x_l)\leqslant \frac{2}{c}\mathcal{A}(x_l)$ for $l$ large enough.
As in \cite[page~3]{HZ87} or \cite[page~109]{HoZe94}, we can use Ascoli-Arzela Theorem to find a $T_b\in\Sigma_{\mathcal{S}}^P$ such that
$$
c_P^j(B_{\mathcal{S}})_{b,m}=T_b\quad\text{ or }\quad c_P^j(B_{\mathcal{S}})_{b,m}=T_b-b.
$$
Note that the map $b\rightarrow c_P^j(B_{\mathcal{S}})_{b,m}$ is nonincreasing. % and $\Sigma_{\mathcal{S}}$ is a discrete set,
We claim that
\begin{equation}\label{e:claim}
\hbox{ the map $b\rightarrow T_b$ must be nonincreasing in both case.}
\end{equation}
In fact, the first case is obvious. Assume that (\ref{e:claim}) does not hold in the second case. Then there exists $b_1<b_2$ such that $T_{b_1}<T_{b_2}$. For each $T\in (T_{b_1},T_{b_2})$, define
$$\Delta_T=\{b\in(b_1,b_2)\,|\,T_b>T\}.$$
Since $T_{b_2}>T$ and $c_P^j(B_{\mathcal{S}})_{b_2,m}\le c_P^j(B_{\mathcal{S}})_{b,m}\le c_P^j(B_{\mathcal{S}})_{b_1,m}$ for any $b\in(b_1,b_2)$, we obtain if $b\in(b_1,b_2)$ sufficiently close to $b_2$, $c_P^j(B_{\mathcal{S}})_{b,m}+b=T_b>T$. Hence $\Delta_T\neq\emptyset$. Set $b_0=\inf\Delta_T$, then $b_0\in[b_1,b_2)$.

Let $\{b_i'\}_i\subset\Delta_T$ satisfy $b_i\downarrow b_0$. Since $ c_P^j(B_{\mathcal{S}})_{b_i',m}\le c_P^j(B_{\mathcal{S}})_{b_0,m}$, we have
$T<T_{b_i'}\le c_P^j(B_{\mathcal{S}})_{b_0,m}+b_i'$, for all $i\in\mathbb{N}$, and thus $T\le T_{b_0}$ by picking $i\rightarrow \infty$.

Suppose that $T< T_{b_0}$. Since $T>T_{b_1}$, thus $b_0\neq b_1$ and so $b_0>b_1$. Thus $b_0\in\Delta_T$. For $b'\in(b_1,b_0)$, $c_P^j(B_{\mathcal{S}})_{b',m}\ge c_P^j(B_{\mathcal{S}})_{b_0,m}$ implies that  $T_{b'}>T$ if $b'$ is close to $b_0$. This contradicts to the definition of $b_0$. Therefore, $T=T_{b_0}$.
This implies $(T_{b_1},T_{b_2})\subset\Sigma_{\mathcal{S}}^P$.
But $\Sigma_{\mathcal{S}}^P=\{k\mathcal{A}(x)\,|\, \,x\;\text{is a}\;\hbox{$P$-symmetric closed characteristic on}\;{\cal S},\;k\in\N \}$.
It easily follows from Lemma~\ref{lem:sik}  that $\Sigma_{\mathcal{S}}^P$ has the empty interior in $(0,\infty)$.
% But we assume  $\Sigma_{\mathcal{S}}^P$ is discrete.
Hence we get a contradiction again. (\ref{e:claim}) is proved.

From (\ref{e:claim}) and $c_P^j(B_{\mathcal{S}})_{b,m}\ge c_P^j(B_{\mathcal{S}})$
we deduce that the second case is impossible for large $b$.
Hence $c_P^j(B_{\mathcal{S}})_{b,m}=T_b$.
Moreover, since $ c_P^j(B_{\mathcal{S}})_{b,m}\rightarrow c_P^j(B_{\mathcal{S}})$ and $\Sigma^P_{\mathcal{S}}$ has the empty interior
as claimed above, we can get a $T\in\Sigma^P_{\mathcal{S}}$ such that
$c_P^j(B_{\mathcal{S}})=T$.

%If $\Sigma_{\mathcal{S}}^P$ is not a discrete set, we can get by approximating $\mathcal{S}$
%by $\widetilde{\mathcal{S}}$ which satisfies the first case and get the results as in \cite{EH89}.

The monotonicity leads to $c_P^j(B_{\mathcal{S}})\geqslant c_P^j(\mathcal{S})$.

In order to prove the converse inequality, for $a>1$, let $\tau_a:\mathbb{R}^{2n}\rightarrow\mathbb{R}$ such that
$$
\tau_a(s)=a\, \text{ if } s\leqslant -\frac{1}{a},\;\; \tau'_a(s)<0\, \text{ if }-\frac{1}{a}<s<-\frac{1}{2a},\; \;
\tau_a(s)=0\, \text{ if } s\geqslant -\frac{1}{2a}.
$$
Define $\gamma_a:\mathbb{R}^{2n}\rightarrow\mathbb{R}$ by
\begin{equation*}
\gamma_a(z)=\left\{
    \begin{array}{ll}
     a\;\; &\text{ if } z\in B_{\mathcal{S}_{-\frac{1}{a}}},\\
     \tau_a(\epsilon)\;\; &\text{ if } z\in\mathcal{S}_{\epsilon}, -\frac{1}{a}<\epsilon\leqslant 0,\\
     0\;\; &\text{ if } z\notin B_{\mathcal{S}}
   \end{array}
   \right.
\end{equation*}
For any $H\in\mathcal{H}(B_{\mathcal{S}})$, define
$H_a(z)=H(z)+\gamma_a(z)$.
 Then when $x$ is a nonconstant $1$-periodic solution of
$$
\dot{x}=X_{H_a}(x),\quad  x(t+\frac{1}{2})=Px(t),\quad x(0)\in B_{\mathcal{S}},
$$
we may deduce that  $x(0)\in\mathcal{S}_{\epsilon}$ for some $\epsilon\in(-1/a,0)$ and
$\mathcal{A}_{H_a}(x)=\tau'_a(\epsilon)-\tau_a(\epsilon)<0$.
Hence the positive critical levels of $\mathcal{A}_{H_a}$ and $\mathcal{A}_H$ are the same. Since $c_P^j(H+s\gamma_a)=\mathcal{A}_{H+s\gamma_a}(x)=\mathcal{A}_H(x)$ for some $x$ and the map $s\rightarrow c_P^j(H+s\gamma_a)$ has to be continuous, it follows that the map $s\rightarrow c_P^j(H+s\gamma_a)$ is constant.
Moreover for every $\widetilde{H}\in\mathcal{H}(\mathcal{S})$ there exist a $H\in\mathcal{H}(B_{\mathcal{S}})$ and a $\gamma_a$
such  that
$\widetilde{H}\leqslant H+\gamma_a$.
Thus, $c_P^j(\widetilde{H})\geqslant c_P^j(H+\gamma_a)=c_P^j(H)\geqslant c_P^j(B_{\mathcal{S}})$.
This implies $c^j_P(\mathcal{S})\geqslant c^j_P(B_{\mathcal{S}})$.
\end{proof}

\section{Proof of Theorem~\ref{th:E} and Theorem~\ref{th:D}}\label{sec:4}
\setcounter{equation}{0}

\begin{proof}[Proof of Theorem~\ref{th:E}]

We first assume that $r_i^2/r^2_j$ ($i\ne j$) are irrational. Then the sequence $d_j(r)$ is strictly monotonic. Define
$$
q(z):=\sum_{i=1}^{n} \frac{x_i^2+y_i^2}{r_i^2}.
$$
It is  the gauge function of $E(r)$. For given $\epsilon>0$ and $l\in\N$,
we pick a smooth increase function $f_{l,\epsilon}:[0,+\infty)\rightarrow[0,+\infty)$ such that %with the following form:
\begin{eqnarray*}
&&f_{l,\epsilon}(s)=0\;\hbox{if $s\le 1$},\qquad f_{l,\epsilon}(s)=(d_l(\sigma_P(r))+\epsilon)s\,\; \hbox{\text{if } s \text{ large enough}},\\
&&f_{l,\epsilon}'(s)<d_l(\sigma_P(r))+2\epsilon\; \forall s\quad\hbox{and}\quad\hbox{$f_{l,\epsilon}'(s)=d_m(\sigma_P(r))$ holds only at $s_m$, for $m=1,\cdots, l$.}
\end{eqnarray*}
%\begin{equation*}
%f_{l,\epsilon}(s)=\left\{
%    \begin{array}{l}
%     0\,\; \text{if } s\leqslant1,\\
%     (d_l(\sigma_P(r))+\epsilon)s\,\; \text{if } s \text{ large enough}.
%   \end{array}
%   \right.
%\end{equation*}
%Morever, $f_{l,\epsilon}'(s)<d_l(\sigma_P(r))+2\epsilon, \forall s$ and $f_{l,\epsilon}'(s)=d_m(\sigma_P(r))$ holds only in $s_m$, for $m=1,\cdots, l$.
Since for every $f_{l,\epsilon}\circ q$, there exists $H\in\mathcal{H}(B)$ such that $f_{l,\epsilon}\circ q\leqslant H$,
 and for any $H\in\mathcal{H}(B)$, there also exists $f_{l,\epsilon}\circ q$ such that  $H\leqslant f_{l,\epsilon}\circ q$,
we obtain
$$
c_P^j(B)=\inf\{c_P^j(f\circ q)\,|\, f=f_{l,\epsilon} \text{ is as  above}\}.
$$
Fix $f=f_{l,\epsilon}$ as above.
The critical points of $\mathcal{A}_{f\circ q}$ are the solutions of the problem
\begin{equation}\label{e:sol}
\left\{
   \begin{array}{l}
     \dot{w}=f'(q(w))J_0\nabla q(w),\\
     w(t+\frac{1}{2})=Pw(t).
   \end{array}
   \right.
\end{equation}
For each solution $w$ of (\ref{e:sol}),  $z(t):=w(t/T)$ is a solution of
\begin{equation}\label{e:1}
\left\{
   \begin{array}{l}
     \dot{z}=J_0\nabla q(z),\\
     z(t+\frac{T}{2})=Pz(t),
   \end{array}
   \right.
\end{equation}
where $T=f'(q(w(0)))$. Let us identify
$\mathbb{R}^{2n}(x_1,\cdots,x_n;y_1,\cdots,y_n)$ with
$\mathbb{R}^2(x_1,y_1)\oplus\cdots\oplus\mathbb{R}^2(x_n,y_n)$,
and write $z(t)=(z_1(t),\cdots,z_n(t))$ with $z_j(t)=(x_j(t),y_j(t))$, $j=1,\cdots,n$. Then
it is easy to compute that $z(t)$ satisfies $\dot{z}=J_0\nabla q(z)$ if and only if
$$
z_j(t)=e^{2t/r_j^2J_0^{(2)}}z_j(0),\quad  1\leqslant j\leqslant n,
$$
where $J_0^{(2)}$ is the standard complex structure on $\mathbb{R}^2$. When $z$ is also required to satisfy the condition
$z(t+\frac{T}{2})=Pz(t)$, by the assumption that $r_i/r_j$ ($i\ne j$) are irrational, we get that the family of solutions for (\ref{e:1}) has the form:
$$
z^{(j)}(t)=(0,\cdots,0, z_j(t),0,\cdots,0),\quad j=1,\cdots,n,
$$
where $z_j(t)=e^{2t/r_j^2J_0^{(2)}}z_j(0)$ with $z_j(0)\in\R^2$ %satisfies $|z_j(0)|^2=r_j^2$ and
has period $T=
(2m+1)\pi r_j^2$ ($j=1,\cdots, n-\kappa$) and $T= 2m\pi r_j^2$ ($j=n-\kappa+1,\cdots, n$) for $m\in\mathbb{N}$.
It follows that $T=f'(q(w(0)))\in\sigma_P(r)$.
 %$2\times 2$ matrix of the symplectic structure.
%$z_i(0)$ belongs to the $i$-th $\mathbb{C}$-component of $\mathbb{R}^{2n}=\mathbb{C}^n$.
%From the condition $z(t+\frac{T}{2})=P z(t)$, $T= (2m+1)\pi r_i^2$, if $i=1,\cdots, n-k$;
%$T= 2m\pi r_i^2$, if $i=n-k+1,\cdots, n$, for $m\in\mathbb{N}$.

 By the construction of $f=f_{l,\epsilon}$,  there exists a $m\in\{1,\cdots,l\}$ such that $q(w)\equiv s_m$.
Then (\ref{e:sol}) is translated into
\begin{equation*}
\left\{
   \begin{array}{l}
     \dot{w}=d_m(\sigma_P(r))J_0\nabla q(w),\\
     w(t+\frac{1}{2})=Pw(t).
   \end{array}
   \right.
\end{equation*}
Since  $\langle \nabla q(z),z\rangle=2q(z)$ and $q(w(t))\equiv q(w(0))$, we get
\begin{eqnarray*}
\mathcal{A}_{f\circ q}(w)&&=\frac{1}{2}\int_0^1\langle -J_0\dot{w}(t), w(t)\rangle dt-\int_0^1 f\circ q(w(t))dt\\
&&=\frac{1}{2}\int_0^1\langle f'(q(w))\nabla q(w), w(t)\rangle dt-\int_0^1 f\circ q(w(t))dt\\
&&=f'(q(w(0)))q(w(0))-f(q(w(0)))\\
&&=f'(s_m) s_m-f(s_m)
\end{eqnarray*}
Therefore, the critical value of $\mathcal{A}_{f_{l,\epsilon}\circ q}$ has the form $f_{l,\epsilon}'(s_m) s_m-f_{l,\epsilon}(s_m)$.

Picking $l$ large enough and fixing an integer $j\le l$, for $m=1,\cdots,j$ let $X_m$ denote the space spanned by solutions of
\begin{equation*}
\left\{
   \begin{array}{l}
     \dot{w}=d_m(\sigma_P(r))J_0\nabla q(w),\\
     w(t+\frac{1}{2})=Pw(t),
   \end{array}
   \right.
\end{equation*}
and put
$$
\xi_j:=E_P^-\oplus E_P^0\oplus_{m=1}^{j} X_m.
$$
Then by Theorem~\ref{prop:dim}, we have $i_P^*(\xi_j)=j$. By choosing $f_{l,\epsilon}$ such that  $s_m$ is  close enough to $1$
for $m=1,\cdots,l$, and $\epsilon>0$ sufficiently small, we can compute to obtain
$$
\sup\mathcal{A}_{f_{l,\epsilon}\circ q}(\xi_j)=f_{l,\epsilon}'(s_j) s_j-f_{l,\epsilon}(s_j)
$$
 and thus $c_P^j(f_{l,\epsilon}\circ q)\leqslant f_{l,\epsilon}'(s_j) s_j-f_{l,\epsilon}(s_j)$.
 Note that $s_j$ can be chosen to be sufficiently close  to $1$. We get that  $c_P^j(f_{l,\epsilon}\circ q)\leqslant d_j(\sigma_P(r))$
and therefore
$$
c_P^j(E(r))\leqslant d_j(\sigma_P(r)).
$$

Next we prove the converse inequality.
By a similar proof to that of \cite[Proposition 4]{EH90} or \cite[Formula (4.2)]{VB82}, we can show that
$c_P^j(f_{l,\epsilon}\circ q)<\infty$ and that  $j\mapsto c_P^j(f_{l,\epsilon}\circ q)$ is strictly increasing.
The choice of $s_j$ implies that $j\mapsto f_{l,\epsilon}'(s_j) s_j-f_{l,\epsilon}(s_j)$ is also strictly increasing.
 Now both
 $$
 \{f_{l,\epsilon}'(s_j) s_j-f_{l,\epsilon}(s_j)\,|\, 1\le j\le l\}\quad\hbox{and}\quad
 \{c_P^j(f_{l,\epsilon}\circ q)\,|\, 1\le j\le l\}
 $$
 are subsets of the critical value set of $\mathcal{A}_{f_{l,\epsilon}\circ q}$
 and the latter is contained in the former.
  Hence
  $$
  c_P^j(f_{l,\epsilon}\circ q)\geqslant f_{l,\epsilon}'(s_j) s_j-f_{l,\epsilon}(s_j),\quad
  j=1,\cdots,l.
  $$
  On the other hand, since
    \begin{eqnarray*}
    c_P^j(f_{l,\epsilon}\circ q)\geqslant  f_{l,\epsilon}'(s_j) s_j-f_{l,\epsilon}(s_j)&\geqslant & d_j(\sigma_P(r))s_j-(d_j\sigma_P(r)+2\epsilon)(s_j-1)\\
    &\geqslant& d_j(\sigma_P(r))+2\epsilon(1-s_j),
    \end{eqnarray*}
  we can take $\epsilon$ so small that $c_P^j(f_{l,\epsilon}\circ q)\geqslant d_j(\sigma_P(r))$ and hence $c_P^j(E(r))\geqslant d_j(\sigma_P(r))$.

Finally, the general case may follow from the continuity of $c_P^j$ and the above special case.
\end{proof}

\begin{proof}[Proof of Theorem~\ref{th:D}]
Firstly, we consider the case $\hat{r}^2/r'^2$ is irrational.
Without loss of generality we can assume  $\hat{r}=r_1$  and $r'=r_{n-k+1}$, and complete the proof in two steps.

{\bf Step 1}. \textsf{Prove the inequality} $c_P^j(D(r))\leqslant  d_j(\sigma_P'(r))$.
 For given $\epsilon>0$ and $l\in\N$, we choose a smooth increase function
 $f_{l,\epsilon}:[0,+\infty)\rightarrow[0,+\infty)$ such that
 \begin{eqnarray*}
 &&f_{l,\epsilon}(s)=0\;\hbox{if}\;s\leqslant1+\epsilon,\quad f_{l,\epsilon}(s)=(d_l(\sigma_P'(r))+\frac{1}{2}\pi) s^2\;\hbox{if}\;s \text{ large enough}\\
  && f_{l,\epsilon}'(s_0)=2 d_l(\sigma_P'(r) s_0\;\;\hbox{and}\;\;s_0>1+\epsilon\Longrightarrow
 f_{l,\epsilon}(s_0)\leqslant\epsilon\;\;\hbox{and}\;\;s_0\leqslant 1+2\epsilon,\\
 &&f_{l,\epsilon}''(s)>0\;\hbox{if}\;s>1+\epsilon.
 \end{eqnarray*}
Define $\varphi: \mathbb{R}^{2n}(x_1,\cdots,x_n;y_1,\cdots,y_n)\equiv \mathbb{R}^2(x_1,y_1)\oplus\cdots\oplus\mathbb{R}^2(x_n,y_n)\to\R$ by
$$
w=(x_1,\cdots,x_n;y_1,\cdots,y_n)\to \varphi(w)=\frac{\sqrt{x_1^2+y_1^2}}{r_1^2}+\frac{\sqrt{x_{n-\kappa+1}^2+y_{n-\kappa+1}^2}}{r_{n+k-1}^2}.
$$
Then $w\in E_P$ is a critical point of $\mathcal{A}_{f_{l,\epsilon}\circ\varphi}$ if and only if $w$ is a solution of
\begin{equation}\label{e:s1}
\left\{
   \begin{array}{l}
     \dot{w}=f_{l,\epsilon}'(\varphi(w))J_0\nabla\varphi(w),\\
     w(t+\frac{1}{2})=Pw(t).
   \end{array}
   \right.
\end{equation}
For each solution $w$ of (\ref{e:s1}),  $z(t):=w(t/T)$ with $T=f_{l,\epsilon}'(\varphi(0))$
 is a solution of
\begin{equation}\label{e:s2}
\left\{
   \begin{array}{l}
     \dot{z}=J_0\nabla\varphi(z),\\
     z(t+\frac{T}{2})=Pz(t).
   \end{array}
   \right.
\end{equation}
As before we write $z(t)=(z_1(t),\cdots,z_n(t))$ with $z_j(t)=(x_j(t),y_j(t))$, $j=1,\cdots,n$. Then
it is easy to compute that $z(t)$ satisfies $\dot{z}=J_0\nabla\varphi(z)$ if and only if
$$
z_j(t)=e^{2t/r_j^2J_0^{(2)}}z_j(0),\quad  j=1,\;n-\kappa+1,\quad z_j(t)\equiv\hbox{const for other $j$},
$$
where $J_0^{(2)}$ is the standard complex structure on $\mathbb{R}^2$. If we also require that $z$ satisfies
$z(t+\frac{T}{2})=Pz(t)$, since $\hat{r}^2/r'^2$ is irrational, we get that the family of solutions for (\ref{e:s2}) has the form:
$$
z^{(j)}(t)=(0,\cdots,0, z_j(t),0,\cdots,0),\quad  1\leqslant j\leqslant n,
$$
where $z_1(t)=e^{2t/r_1^2J_0^{(2)}}z_1(0)$ with $z_1(0)\in\R^2$ %satisfies $|z_1(0)|^2=r_1^2$ and
has period
$T=2(2m-1)\pi r_1^2$ with $m\in\mathbb{N}$,
$z_{n-\kappa+1}(t)=e^{2t/r_{n-\kappa+1}^2J_0^{(2)}}z_{n-\kappa+1}(0)$ with $z_{n-\kappa+1}(0)\in\R^2$   %satisfies $|z_{n-\kappa+1}(0)|^2=r_{n-\kappa+1}^2$ and
has period
$T=4m\pi r_{n-\kappa+1}^2$ with $m\in\mathbb{N}$, and $z_j(t)\equiv\hbox{const}$ for other $j$.

For each integr $l$ large enough and a fixed an integer $j\le l$, and $m=1,\cdots,j$ let $X_m$ denote the space spanned by solutions of
\begin{equation*}
\left\{
   \begin{array}{l}
     \dot{x}=2d_m(\sigma_P'(r) )J_0\nabla\varphi(x),\\
     x(t+\frac{1}{2})=Px(t),
   \end{array}
   \right.
\end{equation*}
 and put $\xi_j=E_P^-\oplus E_P^0\oplus_{m=1}^{j} X_m$.
By Theorem~\ref{prop:dim}, we have $i_P^*(\xi_j)=j$.
As in the proof of Theorem~\ref{th:E}, we deduce that %$\sup\mathcal{A}_{\varphi_{j,\epsilon}}(\xi_j)$  satisfies
 $$\sup\mathcal{A}_{f_{j,\epsilon}\circ\varphi}(\xi_j)=\frac{1}{2}f_{j,\epsilon}'(s)s-f_{j,\epsilon}(s)\leqslant d_j(\sigma_P'(r)) s^2-f_{j,\epsilon}(s)\leqslant  d_j(\sigma_P'(r))(1+2\epsilon)^2.$$
Here $s$ satisfies $f_{j,\epsilon}'(s)=2 d_j(\sigma_P'(r)) s$.
Since $\epsilon>0$ in the construction of $f_{l,\epsilon}$ can be chosen to arbitrarily small,
 we obtain the desired inequality.

\noindent{\bf Step 2}. \textsf{Prove } $c_P^j(D(r))\ge d_j(\sigma_P'(r))$.
 By monotonicity
 it suffices to prove
 \begin{equation}\label{e:s3}
 c_P^j(D(r))\geqslant d_j(\sigma_P'(r))
 \end{equation}
  for $r_i=\hat{r}$\, if $i=1,\cdots,n-\kappa$, and $r_i=r'$\, if $i=n-\kappa+1,\cdots,n$.

  Obviously, $\partial D$ has closed characteristics
  $$
  \R/\Z\ni t\mapsto\gamma_{z_s}(t)=(0,\cdots,0, e^{2\pi tJ_0^{(2)}}z_s,0,\cdots,0)
  $$
  where $z_s=(x_s,y_s)$ has norm $\hat{r}$ ($1\le s\le n-\kappa$) or $r'$ ($n-\kappa+ 1\le s\le n$). The $m$-th iteration of it,
  $(\gamma_{z_s})^m$, is defined by $(\gamma_{z_s})^m(t)= \gamma_{z_s}(mt)$. For $l\in\N$ define $\Sigma^{(2l+1)}$ consisting of
  \begin{eqnarray}\label{e:speciClosed1}
  \Upsilon^{j_1,\cdots,j_{p}}_{z_{s_1},\cdots,z_{s_p}}:=(\gamma_{z_{s_1}})^{j_1}+\cdots+ (\gamma_{z_{s_p}})^{j_p}
  \end{eqnarray}
  where $1\le s_1<s_2<\cdots<s_p\le n-\kappa$ and $j_\nu\in 2\N-1$, $\nu=1,\cdots,p$, satisfy $j_1+\cdots+j_p=2l+1$;
  and also define
  $\Sigma^{(2l)}$ consisting of
  \begin{eqnarray}\label{e:speciClosed2}
  \Upsilon^{j_1,\cdots,j_{p}}_{z_{s_1},\cdots,z_{s_p}}:=(\gamma_{z_{s_1}})^{j_1}+\cdots+ (\gamma_{z_{s_p}})^{j_p}
  \end{eqnarray}
  where $n-\kappa+1\le s_1<s_2<\cdots<s_p\le n$ and $j_\nu\in 2\N$, $\nu=1,\cdots,p$, satisfy $j_1+\cdots+j_p=2l$.
  Then for $\Upsilon^{j_1,\cdots,j_{p}}_{z_{s_1},\cdots,z_{s_p}}$ in (\ref{e:speciClosed1}) (resp. (\ref{e:speciClosed2})) we have
   \begin{eqnarray}\label{e:speciClosed3}
  \mathcal{A}(\Upsilon^{j_1,\cdots,j_{p}}_{z_{s_1},\cdots,z_{s_p}})=(2l+1) \pi \hat{r}^2\quad\hbox{(resp.
  $\mathcal{A}(\Upsilon^{j_1,\cdots,j_{p}}_{z_{s_1},\cdots,z_{s_p}})=2l\pi {r'}^2$)}.
 \end{eqnarray}

Let us pick a sequence of smooth functions, $f_m:\mathbb{R}\rightarrow \mathbb{R}$, $m=1,2,\cdots$, such that
\begin{eqnarray}\label{e:s3.1}
&&f_m(s)=0\; \hbox{if}\; s\leqslant 1+\frac{1}{m},\quad  f_m=(d_m(\sigma_P'(r))+\frac{1}{2})|s|^2\;\hbox{if}\;s>1+\frac{2}{m},\\
&&f_m''(s)>0\;\hbox{if}\;s>1+\frac{1}{m},\label{e:s3.2}\\
&&s_0>1+\frac{1}{m}\;\hbox{and}\;f_m'(s_0)=2d_m(\sigma_P'(r))s_0\;\Longrightarrow \;f_m(s_0)\leqslant\frac{1}{m}.\label{e:s3.3}
\end{eqnarray}
Define $H_m: \mathbb{R}^{2n}(x_1,\cdots,x_n;y_1,\cdots,y_n)\equiv \mathbb{R}^2(x_1,y_1)\oplus\cdots\oplus\mathbb{R}^2(x_n,y_n)\to\R$ by
$$
w=(x_1,\cdots,x_n;y_1,\cdots,y_n)\to
H_m(w)=\sum_{i=1}^{n-\kappa}f_m(\frac{\sqrt{x_i^2+y_i^2}}{\hat{r}^2})+\sum_{i=n-\kappa+1}^nf_m(\frac{\sqrt{x_i^2+y_i^2}}{r'^2}).
$$
 As in the proof of Theorem~\ref{th:E},  we can get
 \begin{equation}\label{e:s4}
 c_P^j(H_m)\rightarrow c_P^j(D(r)), \text{ as }m\rightarrow \infty.
 \end{equation}
Moreover, it follows from Theorem~\ref{th:hig} that $c_P^j(D(r))\in\sigma_P'(r)$.

\begin{claim}\label{cl:closedChar}
 For each $m\in\N$, let $z^{(m)}=(x_1^{(m)},y_1^{(m)},\cdots, x_n^{(m)},y_n^{(m)})$
be a critical point of $\mathcal{A}_{H_m}$ with $\mathcal{A}_{H_m}=c_P^j(H_m)$.
Then the sequence $\{z^{(m)}\}_m$ has a subsequence to converge
 to an orbit on $\partial D(r)$ of form
 (\ref{e:speciClosed1}) or (\ref{e:speciClosed2}).
 \end{claim}

In fact, by (\ref{e:s3.1}) and (\ref{e:s3.3}) we respectively deduce
\begin{eqnarray*}
&&(1+\frac{1}{m})r_i^2\le \sqrt{(x_i^{(m)}(t))^2+(y_i^{(m)}(t))^2}\le (1+\frac{2}{m})r_i^2\;\forall t,\quad i=1,\cdots,n \;\; \hbox{and}\\
&& 0\le H_m(z^{(m)}(t))\le \frac{n}{m}\;\forall t.
\end{eqnarray*}
Since $c_P^j(H_m)=\mathcal{A}_{H_m}(z^{(m)})=\mathcal{A}(z^{(m)})-\int_0^1 H_m(z^{(m)}(t))dt$, we obtain
$$
|c_P^j(H_m)-\mathcal{A}(z^{(m)})|\leqslant \frac{n}{m}.
$$
By a similar argument to the paragraph below (\ref{e:H1})
in the proof of Theorem~\ref{th:hig}, there exists a constant $c$ such that
$$
\frac{1}{2}\langle J_0\frac{d}{dt}{z^{(m)}},z^{(m)}\rangle_{\mathbb{R}^{2n}}\ge c|\frac{d}{dt}{z^{(m)}}|.
$$
This leads to ${\rm length}(z^{(m)})\le c^{-1}\mathcal{A}(z^{(m)})$.
As in \cite[page~3]{HZ87} or \cite[page~109]{HoZe94}  using Ascoli-Arzela Theorem we can find a subsequence converge
 to an orbit on $\partial D(r)$ of form  (\ref{e:speciClosed1}) or (\ref{e:speciClosed2}).

% with action $c_P^j(D(r)$.

If  $j\mapsto c_P^j(D(r))$ is strictly increasing we can prove (\ref{e:s3})
by the same arguments as in the proof of Theorem~\ref{th:E}.
Otherwise, let $\bar{j}$ be the first $j\in\{1,2,\cdots\}$ such that
$c_P^j(D(r))=c_P^{j+1}(D(r))$. Then
\begin{equation}\label{e:s5}
c_P^{i}(D(r))=d_{i}(\sigma_P'(r))\;\forall i\le\bar{j},\quad\hbox{and}\quad
c_P^{\bar{j}}(D(r))=d_{\bar{j}}(\sigma_P'(r))=c_P^{\bar{j}+1}(D(r))
\end{equation}
 and these numbers take values in
 $\{2l\pi r'^2, (2l'+1)\pi\hat{r}^2\,|\, l,l'\in\mathbb{N}\}$.
 Since we have assumed that  $\hat{r}^2/r'^2$ is irrational, sets $\{2l\pi r'^2\,|\, l\in\mathbb{N}\}$ and $\{(2l+1)\pi\hat{r}^2\,|\, l\in\mathbb{N}\}$
 are disjoint, and thus  only one of the following two cases occurs:\\
 \noindent{\bf Case 1.} $c_P^{\bar{j}}(D(r))=(2l+1)\pi\hat{r}^2$ for some $l\in\mathbb{N}$;\\
\noindent{\bf Case 2.} $c_P^{\bar{j}}(D(r))=2l\pi r'^2$  for some $l\in\mathbb{N}$.

Because proofs for two case are similar we only consider Case 1.
%we assume  $c_P^j(d(r))=(2l+1) \pi \hat{r}^2$ for some $l\in\mathbb{N}$. The other case is similar to this case.
%Similarly to the proof of \ref{th:E},
By (\ref{e:s4}) and (\ref{e:s5})
$$
\varrho:=d_{\bar{j}}(\sigma_P'(r))-d_{\bar{j}-1}(\sigma_P'(r))=\lim_{m\rightarrow\infty}(c_P^{\bar{j}}(H_m)-c_P^{\bar{j}-1}(H_m))>0.
$$
%Consider the set $\Sigma$ consisting of trajectories on $\partial D$ of form $\Upsilon_{j_1,\cdots,j_{n-\kappa}}^{z_1,\cdots,z_{n-\kappa}}$ given by
%$$
%[0,1]\ni t\mapsto \Upsilon_{j_1,\cdots,j_{n-\kappa}}^{z_1,\cdots,z_{n-\kappa}}(t)=(e^{2\pi tj_1J_0^{(2)}}z_1,\cdots,  e^{2\pi t j_{n-\kappa}J_0^{(2)}}z_{n-\kappa},0,\cdots,0)
%$$
%where $z_s=(x_s,y_s)$ has norm $\hat{r}$ or zero, $j_s\in 2\mathbb{N}+1$ ($s=1,\cdots,n-\kappa$)
%satisfy $\sum_{\{s\,|\,z_s\neq 0\}} j_s= 2l+1$. Note that
%$\mathcal{A}_{H_m}(\Upsilon_{j_1,\cdots,j_{n-\kappa}}^{z_1,\cdots,z_{n-\kappa}})=\mathcal{A}(\Upsilon_{j_1,\cdots,j_{n-\kappa}}^{z_1,\cdots,z_{n-\kappa}})=(2l+1) \pi \hat{r}^2$
%for all $m\in\N$.

Note that $\Sigma^{(2l+1)}$ consists of finitely many $S^1$-orbits.
We have $i_{S^1,\alpha}(\Sigma^{(2l+1)})=1$. By the continuity of $i_{S^1,\alpha}$ there exists a $S^1$-invariant open neighborhood $U$ of $\Sigma^{(2l+1)}$ such that $i_{S^1,\alpha}(U)=i_{S^1,\alpha}(\Sigma^{(2l+1)})=1$.
 By (\ref{e:s4}), when $m$ is large enough, both $c_P^{\bar{j}}(H_m)$ and $c_P^{\bar{j}+1}(H_m)$ are close enough to $c_P^j(D(r))=c_P^{j+1}(D(r))$.
It follows from Claim~\ref{cl:closedChar} that
  the critical sets of $\mathcal{A}_{H_m}$ on levels $c_P^{\bar{j}}(H_m)$ and $c_P^{\bar{j}+1}(H_m)$ are contained in $U$ for $m$ sufficiently large.
Fix such a $m$. For $\epsilon>0$ and $\delta\in(0, \varrho)$, we can find a $h\in\Gamma$,
$$h(\mathcal{A}_{H_m}^{c_P^{\bar{j}+1}(H_m)+\epsilon}\setminus U)\subset\mathcal{A}_{H_m}^{c_P^{\bar{j}+1}(H_m)+\epsilon-\delta}.$$
By the definition $c_P^{j+1}(H_m)$,
$$
i^*_{S^1,\alpha}(\mathcal{A}_{H_m}^{c_P^{\bar{j}+1}(H_m)+\epsilon})\ge\bar{j}+1.
$$
It follows from this, the subadditivity and the supervariance of $i^*_{S^1,\alpha}$ that
\begin{eqnarray*}
i^*_{S^1,\alpha}(\mathcal{A}_{H_m}^{c_P^{\bar{j}+1}(H_m)+\epsilon-\delta})&\geqslant&
 i^*_{S^1,\alpha}(h(\mathcal{A}_{H_m}^{c_P^{\bar{j}+1}(H_m)+\epsilon}\setminus U))\\
 &\geqslant& i^*_{S^1,\alpha}(\mathcal{A}_{H_m}^{c_P^{\bar{j}+1}(H_m)+\epsilon}\setminus U)\\
 &\geqslant& i^*_{S^1,\alpha}(\mathcal{A}_{H_m}^{c_P^{\bar{j}+1}(H_m)+\epsilon})-i_{S^1,\alpha}(U)\geqslant \bar{j}.
\end{eqnarray*}
This leads to $d_{\bar{j}}(\sigma_P'(r))=c_P^{\bar{j}}(D(r))\leqslant c_P^{\bar{j}+1}(H_m)+\epsilon-\delta$.
But we can $\epsilon>0$ arbitrarily small, $\delta\in(0,\varrho)$ arbitrarily close to $\varrho$.
Combing these with the definition of $\varrho$, we obtain
$$
d_{\bar{j}}(\sigma_P'(r))\leqslant d_{\bar{j}}(\sigma_P'(r))-\varrho\leqslant d_{\bar{j}-1}(\sigma_P'(r)),
$$
which is impossible. So $c_P^{\bar{j}}(D(r)$ is strictly increasing. The proof is complete in this case.
%Thus $c_P^{{j}}(D(r))\geqslant d_{{j}}(\sigma_P'(r))$.

The case that  $\hat{r}^2/r'^2$ is rational
 can be derived from the monotonicity and continuity of $c_P^{\bar{j}}$.
\end{proof}

 % Consider the set $\Sigma$ consisting of trajectories on $\partial D$ of form $\Upsilon_{j_1,\cdots,j_{n-\kappa}}^{z_1,\cdots,z_{n-\kappa}}$ given by
%$$
%[0,1]\ni t\mapsto \Upsilon_{j_1,\cdots,j_{n-\kappa}}^{z_1,\cdots,z_{n-\kappa}}(t)=(e^{2\pi tj_1J_0^{(2)}}z_1,\cdots,  e^{2\pi t j_{n-\kappa}J_0^{(2)}}z_{n-\kappa},0,\cdots,0)
%$$
%where $z_s=(x_s,y_s)$ has norm $\hat{r}$ or zero, $j_s\in 2\mathbb{N}+1$ ($s=1,\cdots,n-\kappa$)
%satisfy $\sum_{\{s\,|\,z_s\neq 0\}} j_s= 2l+1$. Note that
%$\mathcal{A}_{H_m}(\Upsilon_{j_1,\cdots,j_{n-\kappa}}^{z_1,\cdots,z_{n-\kappa}})=\mathcal{A}(\Upsilon_{j_1,\cdots,j_{n-\kappa}}^{z_1,\cdots,z_{n-\kappa}})=(2l+1) \pi \hat{r}^2$
%for all $m\in\N$.

\section{Proof of Theorem~\ref{th:lag}}\label{sec:7}
\setcounter{equation}{0}

We use the method in \cite{BMP19} to compute $c_P^2(D^2\times_LD^2)$.
Slightly modifying the proof of \cite[Proposition~4.1]{BMP19} can lead to:

\begin{lemma}\label{le:c}
Let $D\in\mathcal{B}(\mathbb{R}^{2n})$ be a bounded convex domain, and let $j_D:\mathbb{R}^{2n}\rightarrow \mathbb{R}$ be the Minkowski functional of $D$.
 For any $c>0$, define
$\mathcal{A}_c:E_P\rightarrow\mathbb{R}$ by
$$
\mathcal{A}_c(x):=\mathcal{A}(x)-c\int_0^1j_D(x(t))dt.
$$
If $W\subset E_P$ is a $S^1$-invariant subset with $i_{S^1,\alpha}^*(W)\geqslant k$ such that $\mathcal{A}_c|_W\leqslant 0$, then $c_P^k(D)\leqslant c$.
\end{lemma}

%Notice that if $B$ are $P$-symmetric, the gauge functions of $B$  is $P$-invariant.
%Repeating the proof in \cite{BMP19}, we can get the above two lemma immediately.

\begin{proof}[Proof of Theorem~\ref{th:D}]
Firstly, by  the monotonicity we have
$$
c_P^j(D^2\times_L D^2)\ge c_P^j(B^4(1))\quad\forall j\in\N,
$$
 where $B^4(1)$ is the unit ball in $\mathbb{R}^4$. Specially, $c_P^2(D^2\times_L D^2)\ge c_P^2(B^4(1))=2\pi$.
 %In fact,  we have $c_P^j(D^2\times_L D^2)\geqslant c_P^j(B^4(1))$
% by  the monotonicity of $c_P^j$. Choose a smooth $P$-symmetric convex domain $W$  with a discrete action spectrum $\Sigma_{\partial W}^P$ such that $D^2\times_L D^2\subset W\subset B^4(1)$ and having the same intersection at the boundary with
%$D^2\times_L D^2$ and $B^4(1)$. If $c_P^j(D^2\times_L D^2)=c_P^j(B^4(1))$ for some $j$, then by
%Lemma~\ref{le:com}, implies that there exists a $P$-symmetric closed characteristic contained in
%$$\partial W\cap B^4(1)=\partial(D^2\times_LD^2)\cap B^4(1)=\{(x_1,0,x_2,0)|x_1^2+x_2^2=1\}\cup\{(0,y_1,0,y_2)| y_1^2+y_2^2=1\}.$$
%Since there is no characteristics of $\partial B^4(1)$ in the intersection, we get  $c_P^j(D^2\times_L D^2)>c_P^j(B^4(1))$
%and so  $c_P^2(D^2\times_L D^2)\ge c_P^2(B^4(1))=2\pi$.

Next, let
$$
W:=E_P^-\oplus E_P^0\oplus{\rm span}\{(e^{2\pi it},0),(0, e^{4\pi it})\}.
$$
Then Theorem \ref{prop:dim} yields $i_{S^1,\alpha}^*(W)=2$. For
$$
x(t)=(x_1(t),x_2(t))=(\alpha e^{2\pi it},\beta e^{4\pi it})+w^-+w^0\in W\;\hbox{with}\;\alpha,\beta\in\mathbb{C}, w^-\in E_P^-, w^0\in E_P^0,
$$
since the gauge function $r$ of $D^2\times_L D^2$ is
$$
r((z_1,z_2))=\frac{|z_1|^2+|z_2|^2}{2}+\frac{{\rm Re}(z_1^2+z_2^2)}{2},
$$
 a straightforward computation leads to
\begin{eqnarray*}
\mathcal{A}_c(x)&=&\mathcal{A}(x)-c\int_0^1r(x(t))dt\\
&=&\pi(|\alpha|^2+2|\beta|^2)-\|w^-\|_{\frac{1}{2}}\\
&&-\frac{c}{2}(|\alpha|^2+|\beta|^2)-\frac{c}{2}\|w^0+w^-\|_{L^2}-\frac{c}{2}\int_0^1|{\rm Re}(x_1(t)^2+x_2(t)^2)|dt.
\end{eqnarray*}
We want to find $c$ such that $\mathcal{A}_c(x)\leqslant 0, \forall x\in W$.
For this purpose,   as in the proof of \cite[Theorem~4.6]{BMP19} we can estimate $\mathcal{A}_c(x)$ to  get
$$\mathcal{A}_c(x)\leqslant |\alpha|^2(\pi-\frac{c}{2}-\frac{c^2}{8c+80\pi})+|\beta|^2(2\pi-\frac{c}{2}-\frac{c}{4}+\frac{c^2}{8c+96\pi})-C,$$
where $C>0$ does not depend $\alpha,\beta$.
Thus when
\begin{equation}\label{e:group}
\pi-\frac{c}{2}-\frac{c^2}{8c+80\pi}<0\quad \text{ and }\quad 2\pi-\frac{c}{2}-\frac{c}{4}+\frac{c^2}{8c+96\pi}<0,
\end{equation}
 we can deduce that $\mathcal{A}_c(x)<0$.
Solving the system of inequalities in (\ref{e:group}), we obtain
$c>4\pi\frac{\sqrt{109}-7}{5}.$
Approximating domain of $D^2\times_L D^2$ as in \cite{BMP19} we can use Lemma~\ref{le:c} to derive
$$
c_P^2(D^2\times_L D^2)< 4\pi\frac{\sqrt{109}-7}{5}.
$$
Moreover, by \cite[Proposition 2.2]{BMP19} we know that
the intersection of the action spectrum and $[2\pi,4\pi\frac{\sqrt{109}-7}{5})$ is equal to $\{2\pi, 8\}$.
 Hence $c_P^2(D^2\times_L D^2)\in\{2\pi, 8\}$, which  implies the conclusion.

To compute $c_P^1(D^2\times D^2)$, let us consider
$$
W_1:=E_P^-\oplus E_P^0\oplus{\rm span}\{(e^{2\pi it},0)\}.
$$
By Theorem \ref{prop:dim}, we  get $i_{S^1,\alpha}^*(W)=1$. For $x\in W_1$, repeating the above arguments as in case $\beta=0$, we obtain
$$
\mathcal{A}_c(x)\leqslant |\alpha|^2(\pi-\frac{c}{2}-\frac{c^2}{8c+80\pi})-C.
$$
Thus when $c$ satisfies
\begin{equation}\label{e:group1}
\pi-\frac{c}{2}-\frac{c^2}{8c+80\pi}<0,
\end{equation}
 we have $\mathcal{A}_c(x)<0$.
Solving the inequality in (\ref{e:group1}), we  get
$c>4\pi\frac{\sqrt{41}-4}{5}$.
By approximating domain $D^2\times_L D^2$ as in \cite{BMP19} we may use Lemma~\ref{le:c} to get
$c_P^1(D^2\times_L D^2)< 4\pi\frac{\sqrt{41}-4}{5}$.
Hence
$$
c_P^1(D^2\times_L D^2)\in\bigl[\pi,4\pi\frac{\sqrt{41}-4}{5}\bigr)\cap\Sigma_{\partial (D^2\times D^2)}.
$$
Moreover, $c_P^1(D^2\times_L D^2)\ge c_{EHZ}(D^2\times_L D^2)=4$ by Theorem~\ref{th:min} and \cite[Theorems~1.3, 1.7]{AO14}.
The conclusion follows from these.
\end{proof}

\section{Higher real symmetric Ekeland-Hofer capacities}\label{sec:real}
\setcounter{equation}{0}

%\subsection{$T$-index}

As a complementary for  the symmetrical version of the first Ekeland-Hofer capacity studied
by Jin and the second named author in \cite{JL20},
we outline constructions of higher real symmetric Ekeland-Hofer capacities.

The real part in the standard real symplectic space $(\mathbb{R}^{2n},\omega_0, \tau_0)$ is $L_0:={\rm Fix}(\tau_0)=\{(x,y)\in\mathbb{R}^{2n}\,|\,y=0\}$.
Following \cite{JL20} consider the  Hilbert subspace of the Hilbert space $E$ in (\ref{innerproduct+}),
$$
E_{\tau_0}=\{x\in L^2(S^1;\mathbb{R}^{2n})\,|\,x=\sum_{j\in \mathbb{Z}}e^{2\pi jtJ_0}x_j,\,x_j\in L_0,\,\sum_{j\in\mathbb{Z}}|j||x_j|^2<\infty\}.
$$
%with the inner product
%$$
%\langle x,y\rangle_{1/2}:=\langle x_0,y_0\rangle+2\pi\sum_{k\in\mathbb{Z}}|k|\langle x_k,y_k\rangle
%$$
%and the induced norm $\|\cdot \|_{1/2}$. % on $E_{\tau_0}$ givento be $\parallel x\parallel_{1/2}=\langle x,x\rangle_{1/2}$.
Denote $S=\{x\in E_{\tau_0}\,|\,\|x\|_{1/2}=1\}$.
Define a $\mathbb{Z}_2$-action $T=\{T_0, T_1\}$ on $E_{\tau_0}$ as follows:
$$
T_0:E_{\tau_0}\rightarrow E_{\tau_0}, x(t)\rightarrow x(t),\quad T_1:E_{\tau_0}\rightarrow E_{\tau_0}, x(t)\rightarrow x(t+\frac{1}{2}).
 $$
 Let $\mathfrak{H}$ be the set of $T$-equivariant maps from $E_{\tau_0}$ to itself, and let $F$ denote
the fixed point set of this $\mathbb{Z}_2$ action, i.e.,
$$
F=\{x\in E_{\tau_0}\,|\,x=\sum_{j\in \mathbb{Z}}e^{2\pi jtJ_0}x_j,x_j\in L_0,\, x_j=0,\, \text{ if } j \text{ is odd }\}.
$$
 There exists an  orthogonal splitting of $E_{\tau_0}$,
$E_{\tau_0}=F\oplus G$, where
$$
G=\{x\in E_{\tau_0}\,|\,x=\sum_{j\in\mathbb{Z}} e^{2\pi jtJ_0}x_j, x_j\in L_0, x_j=0,\text{ if } j \text{ is even} \}.
$$
Denote by  $P_1$ and $P_2$ the orthogonal projections to $F$ and $G$, respectively.
 Let
$$
\mathcal{E}_T=\{A\subset E_{\tau_0}\,|\,A \text{ is closed and } T\text{-invariant}\}.
$$
There exists a natural index $i_{\tau_0}:\mathcal{E}_T\to\N\cup\{0,\infty\}$ (cf. \cite{AS89} for example) given by
$$
i_{\tau_0}(A)=\min\{k\in\N\,|\,\exists f\in C(A,\mathbb{R}^k\setminus\{0\})\text{ satisfying }f(T_1 x)=-f(x)\;\forall x\in A\}
$$
for each nonempty  $A\in\mathcal{E}_T$, where $i_{\tau_0}(A)$ is defined to be $\infty$
if there is no such $k\in \N$. Of course, $i_{\tau_0}(A)=0$ if $A=\emptyset$.
Notice that $A\cap F\neq \emptyset $ implies $i_{\tau_0}(A)=\infty$ (since $f(x)=f(T_1 x)=-f(x)$ for $x\in A\cap F$).
$i_{\tau_0}$ possess the properties of the index stated in \cite[Definition~1.1]{VB82} (see  \cite{AS89} for proofs).
We have also

\begin{proposition}
$i_{\tau_0}$ %is a normalized $T$-index on $\mathcal{E}_T$ and
satisfies $1$-dimension property.
\end{proposition}
\begin{proof}
%It's obvious that (i) and (ii) is hold. Consider $A,B\in\mathcal{E}_T$ and $i_{\tau_0}(A)=m<\infty, i_{\tau_0}(B)=k<\infty$. Then there exist $f:A\rightarrow\mathbb{R}^{m}\setminus\{0\}$ satisfying $f(T_1 x)=-x, \forall x\in A$ and $g:B\rightarrow\mathbb{R}^k\setminus\{0\}$  satisfying $ g(T_1 x)=-x,\forall x\in B$. We can get extensions $\widetilde{f}:E_{\tau_0}\rightarrow\mathbb{R}^m, \widetilde{g}:E_{\tau_0}\rightarrow \mathbb{R}^k$ of $f$ and $g$ such that $\widetilde{f}(T_1x)=-\widetilde{f}(x),\widetilde{g}(T_1 x)=-\widetilde{g}(x)$. Denote
%$$h:A\cup B\rightarrow \mathbb{R}^{m+k},x\mapsto (\widetilde{f}(x),\widetilde{g}(x)),$$
%Thus, $h(x)\neq (0,0), \forall x\in A\cup B$. So $i_{\tau_0}(A\cup B)\leqslant i_{\tau_0}(A)+i_{\tau_0}(B)$.
%
%Next, we prove (iv). For $A\in\mathcal{E}_T,h\in\mathcal{H}$, we assume that $i(h(A))=k<\infty$. Then there exists $f:h(A)\rightarrow\mathbb{R}^k\setminus\{0\}$ satisfying $f(T_1 x)=-f(x)$. Then $f\circ h: A\rightarrow \mathbb{R}^k\setminus\{0\}$ satisfying $f\circ h(T_1 x)=f(T_1 h(x))=-f(h(x))$. Thus, $i_{\tau_0}(A)\leqslant i_{\tau_0}(h(A))$.
%
%(v) can be obtained from \cite[Proposition2.1]{AS89}.
% Finally, we prove $i_{\tau_0}$ satisfies $1$-dimension property.
For a $T$-invariant subspace $V^k$ with ${\rm dim}(V^k)=k$, if $V^k\cap F=\{0\}$, then $V^k\cap S$ is the sphere of $V^k$, ${\rm dim}(V^k)=k$.
It is obvious that $i_{\tau_0}(V^k\cap S)\leqslant k$.

% Consider orthogonal splitting of $E$: $E_{\tau_0}=F\oplus G$,
 Note that $P_2\in\mathfrak{H}$. Denote $V=P_2(V^k\cap S)$. For $l<k$, and $f:V\rightarrow \mathbb{R}^l\subset \mathbb{R}^k$ satisfying $f(T_1 x)=-f(x)$.
 Since $V\subset G$, $f(T_1 x)=f(-x)=-f(x)\;\forall x\in V$, that is, $f$ is odd. By Borsuk theorem,
  there must be some $x\in V$ such that $f(x)=0$. So $i_{\tau_0}(V^k\cap S)\leqslant i_{\tau_0}(P_2(V^k\cap S))\leqslant k$. This implies $i_{\tau_0}(V^k\cap S)=k$.
\end{proof}

%From \cite{AS89}, we can get
% Notice that if a $T$-orbit $[x]$ in $X$ satisfies $[x]\cap F=\emptyset$,
% then $[x]$ has only two points. Thus $i_{\tau_0}([x])=1$. So $i$ is normalized.

Consider the orthogonal composition of $E_{\tau_0}$
$E_{\tau_0}=E_{\tau_0}^-\oplus E_{\tau_0}^0\oplus E_{\tau_0}^+$,
where $E_{\tau_0}^0=L_0={\rm Fix}(\tau_0)=\{(x,y)\in\mathbb{R}^{2n}|y=0\}$, and
$$
E_{\tau_0}^-=\{x\in E_{\tau_0}\,|\,x=\sum_{j<0}x_j e^{2\pi jtJ_0}\}\quad\hbox{and}\quad E_{\tau_0}^+=\{x\in E_{\tau_0}\,|\, x=\sum_{j>0}x_j e^{2\pi jtJ_0}\}.
$$
Denote $P^+, P^-, P^0$ by the orthogonal projections to $E_{\tau_0}^+, E_{\tau_0}^-, E_{\tau_0}^0$, respectively.
%There, every $x\in E_{\tau_0}$ has the unique decomposition
%$x=x^-+x^0+x^+\in E_{\tau_0}^-\oplus E_{\tau_0}^0\oplus E_{\tau_0}^+$$
%where $x^+=P^+x, x^0=P^0x$ and $x^-=P^-x$.
Consider the subgroup of $\mathfrak{H}$,
$$
\Gamma_T:=\{h:E_{\tau_0}\rightarrow E_{\tau_0}\,|\,h(x)=e^{\gamma^+(x)}P^+(x)+P^0(x)+e^{\gamma_-(x)}P^-(x)+K(x)\},
$$
where i) $K:E_{\tau_0}\rightarrow E_{\tau_0}$ is a $T$-equivariant continuous map, and maps bounded set to precompact set, ii)
 $\gamma^+,\gamma^-:E_{\tau_0}\rightarrow \mathbb{R}^+$ is $T$-invariant continuous function, and maps bounded set to bounded set, and iii)
 there exists a constant $c>0$ such that if $x\in E_{\tau_0}$ satisfies $\mathcal{A}(x)\leqslant 0$ or $\|x\|_{1/2}\geqslant c$,
 then $\gamma^+(x)=\gamma^-(x)=0$ and $K(x)=0$.

For $\xi\in\mathcal{E}_T$, we define the pseudoindex of $i_{\tau_0}$ relative to $\Gamma_T$:
$$
i_{\tau_0}^*(\xi):=\inf\{i(h(\xi)\cap S\cap E_{\tau_0}^+)\,|\,h\in\Gamma_T\}.
$$
%We have the following result about $i^*_{\tau_0}$.
%\begin{theorem}\label{th:dim2}
%For a $T$-invariant subspace $X$ of $G$ with ${\rm dim}(X)=p$, we have
%$$ i^*_{\tau_0}(F\oplus X)=k. $$
%
%\end{theorem}
%This is a result of \cite[Proposition2.1]{VB82} in the case of $d=1, H^+=G, H^-=F$.
%\subsection{The definition of higher real symmetric Ekeland-Hofer capacity}
Let $\mathscr{H}(\mathbb{R}^{2n})$ be  the set of nonnegative  function $H\in C^{\infty}(\mathbb{R}^{2n})$ satisfying the following condition:
\begin{description}
\item[(H1)] $H(z)=H(\tau_0 z), \forall z\in\mathbb{R}^{2n}$;
\item[(H2)] There is an open set $U\subset\mathbb{R}^{2n}$ such that $H|_U\equiv 0$;
\item[(H3)] When $|z|$ is large enough, $H(z)=a|z|^2$, where $a>\pi$ and $a\notin \mathbb{N}\pi$.
\end{description}

Let $B\subset\mathbb{R}^{2n}$ be bounded, $\tau_0$-invariant and $B\cap L_0\neq \emptyset$, and let
$\mathcal{F}(\mathbb{R}^{2n}, B)$ consist of $H\in\mathscr{H}(\mathbb{R}^{2n})$ such that $H=0$ in a neighborhood of $\overline{B}$.
%$$
%\mathcal{F}(\mathbb{R}^{2n}, B)=\{H\in\mathscr{H}(\mathbb{R}^{2n})\,|\,H=0\text{ in a neighborhood of } \overline{B}\}.
%$$
For $H\in\mathscr{H}(\mathbb{R}^{2n})$, define
$$
\mathcal{A}_H:E_{\tau_0}\rightarrow\mathbb{R},\;x\mapsto
\frac{1}{2}\int_0^1\langle -J_0\dot{x},x\rangle dt-\int_0^1 H(x(t))dt
$$
and
$$
c_{\tau_0}^j(H):=\inf\{\sup\mathcal{A}_H(\xi)\,|\,\xi\in\mathcal{E}_T\;\&\; i_{\tau_0}^*(\xi)\geqslant j\},\quad \forall j\in\N.
$$
As in the proof of Proposition~\ref{prop:H-cap} we have
\begin{proposition}
There is $\beta>0$ such that $0<\beta\leqslant c_{\tau_0}^1(H)\leqslant c_{\tau_0}^2(H)\leqslant\cdots\leqslant c_{\tau_0}^k(H)$.
\end{proposition}
%\begin{proof}
%For $H\in\mathcal{F}(\mathbb{R}^{2n},B)$, $H=0$ in a $\tau_0$-invariant open set $U\subset\mathbb{R}^{2n}$ and pick $x_0\in U\cap L_0$. Arguing as in \cite{EH90} or \cite{JL20}, find $\epsilon>0$ small such that
%$$\mathcal{A}_H|_{x_0+\epsilon S_1}\geqslant\beta>0,$$
%where $S_1=S\cap G$. Pick $h\in\Gamma'$ satisfying
%$$h(S_1)=x_0+\epsilon S_1.$$
%Hence  $i_{\tau_0}^*(\xi)\geqslant 1$ implies $h^{-1}(\xi)\cap S_1\neq\emptyset$.
%Thus, $\emptyset\neq\xi\cap h(S_1)=\xi\cap(x_0+\epsilon S_1).$
%We can find $c_{\tau_0}^k(H)\geqslant c_{\tau_0}^1(H)\geqslant \beta>0$.
%\end{proof}

As before, %in Section~\ref{var}, we can know for $H\in\mathscr{H}(\mathbb{R}^{2n})$,
%$\mathcal{A}_H$ satisfies (P.S.) condition. As the proof of Proposition \ref{pro:cri}, we can get that
if $c_{\tau_0}^j(H)<\infty$, then $c_{\tau_0}^j(H)$ is critical value of $\mathcal{A}_H$,
and $c_{\tau_0}^k(H_2)\geqslant c_{\tau_0}^k(H_1)$  for any two $H_1, H_2\in\mathscr{H}(\mathbb{R}^{2n})$ with $H_1\geqslant H_2$,
$k=1,2,\cdots$. Let
$\mathcal{B}_{\tau_0}(\mathbb{R}^{2n})$
be the set of $\tau_0$-invariant subset $B\subset\mathbb{R}^{2n}$ with $B\cap L_0\neq\emptyset$.
For each $j=1,2,\cdots$, if $B\in \mathcal{B}_{\tau_0}(\mathbb{R}^{2n})$ is bounded we call
$$
c_{\tau_0}^j(B):=\inf_{H\in\mathcal{F}(\mathbb{R}^{2n},B)}c_{\tau_0}^j(H)
$$
the \textsf{$j$-th real symmetric Ekeland-Hofer capacity}, and if $B\in \mathcal{B}_{\tau_0}(\mathbb{R}^{2n})$ is unbounded we define
$$
c_{\tau_0}^j(B):=\sup\{c_{\tau_0}^j(B')\,|\,B'\subset B\text{ is bounded and }\; B'\in \mathcal{B}_{\tau_0}(\mathbb{R}^{2n})\}.
$$
It is easy to verify that $c_{\tau_0}^j$ also satisfies the corresponding properties in
Proposition~\ref{prop:property}.
%As $P$-symmetric Ekeland-Hofer capacities, we have
%\begin{proposition}
%$\forall j\in\mathbb{N}$, $c_{\tau_0}^j$ has the following property:
%\begin{description}
%\item[(i)(Monotonicity)]If $B_1,B_2\in\mathcal{B}_{\tau_0}(\mathbb{R}^{2n})$ and $B_1\subset B_2$, then $c_{\tau_0}^j(B_1)\leqslant c_{\tau_0}^j(B_2)$;
%\item[(ii)(Conformality)] $c_{\tau_0}^j(\lambda B)=\lambda^2c_{\tau_0}^j(B),\forall\lambda\in\mathbb{R},\forall B\in\mathcal{B}_{\tau_0}(\mathbb{R}^{2n})$;
%\item[(iii)] If $h\in{\rm Symp}(\mathbb{R}^{2n})$ is $\tau_0$-equivariant, then
%$$c_{\tau_0}^j(h(B))=c_{\tau_0}^j(B);$$
%\item[(iv)]$c_{\tau_0}^j:\mathcal{B}_{\tau_0}(\mathbb{R}^{2n})\rightarrow\mathbb{R}$ is continuous with respected to Hausdroff distance on $\mathcal{B}_{\tau_0}(\mathbb{R}^{2n})$.
%\end{description}
%\end{proposition}
By a similar argument to the proof of \cite[Theorem 1.11]{JL20}, we have the following representation formula for $c_{\tau_0}^1$.

\begin{theorem}\label{th:repre}
Let $D\subset\mathbb{R}^{2n}$ be a $\tau_0$-invariant convex bounded domain with $C^{1,1}$ boundary $\mathcal{S}=\partial D$ and contain a fixed point of $\tau_0$. Then
$$
c_{\tau_0}^1(D)=\min\{\mathcal{A}(x)>0\,|\, x\text{ is a }\tau_0\text{-brake closed characteristic on } \mathcal{S}\},
$$
where by the definition of \cite{JL20} a $\tau_0$ brake closed characteristic on $\mathcal{S}$ is a closed characteristic $z:\mathbb{R}/T\mathbb{Z}\rightarrow \mathcal{S}$ on $\mathcal{S}$ satisfying $z(T-t)=\tau_0 z(t)$.
Moreover, if both $\partial D$ and $D$ contain fixed points of $\tau_0$, then
$c_{\tau_0}^1(D)=c_{\tau_0}^1(\partial D)$.
\end{theorem}

It follows from this result and \cite[Theorem 1.3]{JL20} that $c_{\tau_0}^1$ and the real symmetric Ekeland-Hofer capacity $c_{\rm EHZ,\tau_0}$ defined in \cite{JL20}
coincide on any $\tau_0$-invariant convex domain $D\subset\mathbb{R}^{2n}$.

%Note that this can imply $c_1^{\tau_0}$ is equal to  in the domain satisfying the assumption of  the above theorem.
%Notice that we do not give a dimension formula like Theorem \ref{prop:dim} for $i_{\tau_0}^*$, so we only give the definition of $c_{\tau_0}^j$ in form and can not compute specify examples for it.

%If a hypersurface $\mathcal{S}$ of restricted contact type in $\mathbb{R}^{2n}$ is symmetric with respect to the origin
%one can always choose an odd  Liouville vector field $\eta$ on $\mathbb{R}^{2n}$ such that $\mathcal{L}_{\eta}\omega_0=\omega_0$
% and $\eta$ points transversely outward at $\mathcal{S}$ (\cite{EH890}).
%Arguing as in the proofs of \cite[Theorem~1.18]{JL20} or Theorem~\ref{th:hig}
%we can also show that $c_{\tau_0}^j(\mathcal{S})=c_{\tau_0}^j(B_\mathcal{S})$ is the multiplicity of
%the action of some $\tau_0$ brake closed characteristic on $\mathcal{S}$.

\begin{tabular}{l}
Laboratory of Mathematics and Complex Systems (Ministry of Education), \\
School of Mathematical Sciences, Beijing Normal University,\\
 Beijing 100875, People's Republic of China\\
E-mail address: shikun@mail.bnu.edu.cn,\hspace{5mm}gclu@bnu.edu.cn\\
\end{tabular}

\end{document}